\documentclass[11pt]{amsart}

\usepackage{amsmath, amsthm, amssymb}
\usepackage{graphicx, array}
\usepackage{longtable}
\usepackage{epstopdf}
\usepackage{mathrsfs}

\theoremstyle{definition}

\newtheorem{theorem}{Theorem}[section]
\newtheorem{proposition}{Proposition}
\newtheorem{lemma}{Lemma}
\newtheorem{corollary}{Corollary}
\newtheorem{definition}{Definiton}
\newtheorem{remark}{Remark}
\newtheorem{Ex}{Example}
\newtheorem{conjecture}[theorem]{Conjecture}

\title{The word problem for some classes of Adian inverse semigroups}

\author{Muhammad Inam}
\email{s-minam1@math.unl.edu}

\address{Department of Mathematics\\
University of West Georgia\\
Carrollton, Georgia 30118 USA}

\begin{document}

\date{\today}

\begin{abstract}
We show that all of the Sch\"{u}tzenberger complexes of an Adian inverse semigroup are finite if the Sch\"{u}tzenberger complex of every positive word is finite. This enables us to solve the word problem for certain classes of Adian inverse semigroups (and hence for the corresponding Adian semigroups and Adian groups).
\end{abstract}

\maketitle

\section
{Introduction}

Groups and semigroups are often presented by a pair $\langle X|R\rangle$,where $X$ denotes the set of set of generators and $R$ denotes the set of defining relations.  If $R=\{(u,v)|u,v\in X^+\}$, then the presentation $\langle X|R\rangle$ is called a \textit{positive presentation}. We consider positive presentations throughout this paper. We can construct two undirected graph corresponding to a positive presentation $\langle X|R\rangle$. These graphs are called the \textit{left graph} and the \textit{right graph} of the presentation and are denoted by $LG\langle X|R\rangle$ and $RG\langle X|R\rangle$ respectively.  The vertices of the $LG\langle X|R\rangle$ and the $RG\langle X|R\rangle$ are labeled by the elements of the set $X$. In the $LG\langle X|R\rangle$, we obtain an edge by joining the vertex labeled by the prefix letter of $u$ with the vertex labeled by the prefix letter of $v$ for all $(u,v)\in R$.  The $RG\langle X|R\rangle$ is constructed dually by joining the vertex labeled by the suffix letter of $u$ with the vertex labeled by suffix letter of $v$ for all $(u,v)\in R$. If there is no cycle in the left and the right graph of a presentation then the presentation is called a \textit{cycle free presentation} or an \textit{Adian presentation}. These presentation were first studied by S. I. Adian \cite{AD}, where is shown that the finitely presented Adian semigroups $Sg\langle X|R\rangle$ embeds in the corresponding Adian $Gp\langle X|R\rangle$. Latter in \cite{RM} John H. Remmers generalized this result to any Adian presentation and proved that an Adian semigroup $Sg\langle X|R\rangle$ embeds in the corresponding Adian group $Gp\langle X|R\rangle$. Unless stated otherwise, throughout this paper we consider our presentations to be Adian presentations.

A semigroup $S$ is called an \textit{inverse semigroup} if for every element $a\in S$ there exists a unique element $b\in S$ such that $aba=a$ and $bab=b$. This unique element $b$ is called the \textit{inverse of $a$} and denoted by $a^{-1}$. The \textit{natural partial order} on the elements of an inverse semigroup $S$ is defined as $a\leq b$ if and only if $a=aa^{-1}b$, for some $a,b\in S$.  The congruence relation $\sigma$ on $S$ is defined as for $a,b\in S$, $a\sigma b$ if and only if there exists an element $c\in S$ such that $c\leq a b$. It turns out that the $\sigma$ is the minimal group congruence on $S$. So, $S/\sigma$ is the maximal group homomorphic image of $S$ and if $S$ is presented by a presentation $Inv\langle X|R\rangle$, then $S/\sigma$ is isomorphic corresponding group $Gp\langle X|R\rangle$. Detailed proofs of these facts about inverse semigroups are provided in the text \cite{IS}. In order to be consistent with most of the literature about inverse semigroups, we abuse the notation of $\sigma$ and we also denote the natural homomorphism from $inv\langle X|R\rangle$ to $Gp\langle X|R\rangle$ by $\sigma$. An inverse semigroup $S$ is called \textit{$E$-unitary} if $\sigma^{-1}(1)=\{e| e^2\underset{S}{=}e\}$.

 J. B. Stephen \cite{SG} introduced the notion of \textit{Sch\"{u}tzenberger graphs} to solve the word problem for inverse semigroups. If $M=Inv\langle X|R\rangle$ is an inverse semigroup then we may consider the corresponding Cayley graph $\Gamma(M,X)$. The vertices of this graph are labeled by the elements of $M$ and there exists a directed edge labeled by $x\in X\cup X^{-1}$ from the vertex labeled by $m_1$ to the vertex labeled by $m_2$ if $m_2\equiv m_1x$.   The Cayley graph $\Gamma(M,X)$ is not necessarily strongly connected, unless $M$ happens to be a group, therefore there may not be an edge labeled by $x^{-1}$ from $m_2$ to $m_1$. The strongly connected components of $\Gamma(M,X)$ are called the \textit{Sch\"{u}tzenberger graphs} of $M$. For any word $u\in (X\cup X^{-1})^*$ the strongly connected component of $\Gamma(M,X)$ that contains a vertex labeled by $u$ is the \textit{Sch\"{u}tzenberger graph} of $u$ and it is denoted by $S\Gamma(M,X,u)$. The vertices of $S\Gamma(M,X,u)$ are labeled by the elements of $\mathscr{R}$-class of $u$, i.e., $R_u=\{m\in M| mm^{-1}\underset{M}{=}uu^{-1}\}$, because if $x\in X\cup X^{-1}$ labels an edge from a vertex labeled $m\in M$ to a vertex labeled by $mx$, then there exists an edge labeled by $x^{-1}$ from $mx$ to $m$  in $\Gamma(M,X)$ if and only if $m\mathscr{R}(mx)$.

 There exists a natural graph morphism (not necessarily injective) from a Sch\"{u}tzenberger graph of $M$ to the Cayley graph of $M$. Meakin showed that each Sch\"{u}tzenberger graph of $M$ embeds in the Cayley graph of $M$ if and only if $M$ is $E$-unitary.

 It is useful to consider \textit{Sch\"{u}tzenberger automaton} $(uu^{-1},S\Gamma(M,X,u),u)$ with initial vertex $uu^{-1}\in M$, terminal vertex $u\in M$ and set of states are the vertices of $S\Gamma(M,X,u)$. The language accepted by this automaton is

 \begin{center}
 $L(u)=\{v\in M|v$ labels a path  from $uu^{-1}$ to $u$ in $S\Gamma(M,X,u)\}$.

 \end{center}
Here $u$ and $v$ are regarded as elements of both $(X\cup X^{-1})^*$ and $M$. So, $L(u)$ is a subset of $M$.

The following result of Stephen \cite{SG} plays a key role in solving the word problem for inverse semigroups.
\begin{theorem}\label{Stephen's thm}
Let $M=Inv\langle X|R\rangle$ and let $u,v\in (X\cup X^{-1})^*$ Then\\
\begin{enumerate}

\item $L(u)= \{v\in M|v\geq u$ in the natural partial order on $M\}$.

\item $u\underset{M}{=} v $$\Leftrightarrow $$L(u)=L(v)$ $ \Leftrightarrow$ $ u\in L(v)$ and $v\in L(u)$$ \Leftrightarrow$$ (uu^{-1},S\Gamma(M,X,u),$\\$u)$ and $(vv^{-1},S\Gamma(M,X,v),v)$ are isomorphic as birooted edge-labeled graphs.

\end{enumerate}
\end{theorem}

 We briefly describe the iterative procedure described by Stephen \cite{SG} for building a Sch\"{u}tzenberger graph. Let $Inv\langle X|R\rangle$  be a presentation of an inverse monoid.

   Given a word $u=a_1a_2...a_n\in (X\cup X^{-1})^*$, the \textit{linear graph} of $u$  is the birooted inverse word graph $(\alpha_u,\Gamma_u,\beta_u)$ consisting of the set of vertices

   \begin{center}

   $V((\alpha_u,\Gamma_u,\beta_u))=\{\alpha_u,\beta_u,\gamma_1,...,\gamma_{n-1}\}$

   \end{center}

   and edges

   \begin{center}

   $(\alpha _u,a_1, \gamma _1),(\gamma _1,a_2,\gamma _2),..., (\gamma _{n-2},a_{n-1},\gamma _{n-1}),(\gamma _{n-1},a_n,\beta _u)$,

   \end{center}

    together with the corresponding inverse edges.

    Let $(\alpha , \Gamma ,\beta )$ be a birooted inverse word graph over $X\cup X^{-1}$. The following operations may be used to obtain a new birooted inverse word graph $(\alpha ',\Gamma ',\beta ')$:

    $\bullet$ \textbf{Determination} or \textbf{folding:} Let $(\alpha,\Gamma,\beta)$ be a birooted inverse word graph with vertices $v,v_1,v_2$, with $v_1\neq v_2$, and edges $(v,x,v_1)$ and $(v,x,v_2)$ for some $x\in X\cup X^{-1}$.

    Then we obtain a new birooted inverse word graph $(\alpha',\Gamma',\beta')$ by taking the quotient of $(\alpha,\Gamma,\beta)$ by the equivalence relation which identifies the vertices $v_1$ and $v_2$ and the two edges. In other words, edges with the same label coming out of a vertex are folded together to become one edge.

    $\bullet$ \textbf{Elementary $\mathscr{P}$-expansion:} Let $r=s$ be a relation in $R$ and suppose that $r$ can be read from $v_1$ to $v_2$ in $\Gamma$, but $s$ cannot be read from $v_1$ to $v_2$ in $\Gamma$. Then we define $(\alpha',\Gamma',\beta')$ to be the quotient of $\Gamma \cup (\alpha _s,\Gamma_s,\beta_s)$ by the equivalence relation which identifies vertices $v_1$ and $\alpha_s$ and vertices $v_2$ and $\beta_s$. In other words. we  ``sew" on a linear graph for $s$ from $v_1$ to $v_2$ to complete the other half of the relation $r=s$.

    An inverse word graph is \textit{deterministic} if no folding can be performed and \textit{closed} if it is deterministic and no elementary expansion can be performed over a presentation $\langle X|R\rangle$. Note that given a finite inverse word graph it is always possible to produce a determinized form of the graph, because determination reduces the number of vertices. So, the process of determination must stop after finitely many steps, We note also that the process of folding is confluent \cite{SG} .

    If $(\alpha_1,\Gamma_1, \beta_1)$ is obtained from $(\alpha,\Gamma,\beta)$ by an elementary $\mathscr{P}$-expansion, and $(\alpha_2,\Gamma_2,\beta_2)$ is the determinized  form of $(\alpha_1,\Gamma_1,\beta_1)$, then we write $(\alpha,\Gamma,\beta)$\\$\Rightarrow (\alpha_2,\Gamma_2,\beta_2)$ and say that $(\alpha_2,\Gamma_2,\beta_2)$ is obtained from $(\alpha, \Gamma,\beta)$ by a  \textit{$\mathscr{P}$-expansion}. The reflexive and transitive closure of $\Rightarrow$ is denoted by $\Rightarrow ^*$.

    For $u\in (X\cup X^{-1})^*$, an \textit{approximate graph} of $(uu^{-1}, S\Gamma(u), u)$ is a birooted inverse word graph $A=(\alpha,\Gamma,\beta)$ such that $u\in L[A]$ and $y\geq u$ for all $y\in L[A]$. Stephen showed in \cite{SG} that the linear graph of $u$ is an approximate graph of $(uu^{-1}, S\Gamma(u), u)$. He also proved the following:

    \begin{theorem}\label{closure}
    Let $u\in (X\cup X^{-1})$ and let $(\alpha,\Gamma,\beta)$ be an approximate graph of $(uu^{-1},S\Gamma(u), u)$. If $(\alpha,\Gamma,\beta)\Rightarrow^*(\alpha',\Gamma',\beta')$ and $(\alpha',\Gamma',\beta')$ is closed , then $(\alpha',\Gamma',\beta')$ is the Sch\"{u}tzenberger graph of $u$, $(uu^{-1},S\Gamma(u), u)$.
    \end{theorem}

    In \cite{SG}, Stephen showed that the class of all birooted inverse words graphs over $X\cup X^{-1}$ is a co-complete category  and that the directed system of all finite $\mathscr{P}$-expansions of a linear graph of $u$ has a direct limit. Since the directed system includes all possible $\mathscr{P}$-expansions, this limit must be closed. Therefore, by \ref{closure}, the Sch\"{u}tzenberger graph is the direct limit.

    \textbf{Full $\mathscr{P}$- expansion (a generalization of the concept of $\mathscr{P}$-\\ expansion):} A full $\mathscr{P}$-expansion of a birooted inverse word graph $(\alpha,\Gamma,\beta)$ is obtained in the following way:

    $\bullet$ Form the graph $(\alpha',\Gamma',\beta')$, which is obtained from $(\alpha,\Gamma,\beta)$ by performing all possible elementary $\mathscr{P}$-expansions of $(\alpha,\Gamma,\beta)$, relative to $(\alpha,\Gamma,\beta)$. We emphasize  that an elementary $\mathscr{P}$-expansion may introduce a path labeled by one side of relation in $R$, but we do not perform an elementary $\mathscr{P}$-expansion that could not be done to $(\alpha,\Gamma,\beta)$ when we do a full $\mathscr{P}$-expansion.

    $\bullet$ Find the determinized form $(\alpha_1,\Gamma_1,\beta_1)$, of $(\alpha',\Gamma',\beta')$.

    The birooted inverse word graph $(\alpha_1,\Gamma_1,\beta_1)$ is called the full $\mathscr{P}$-expansion of $(\alpha,\Gamma,\beta)$. We denote this relationship by $(\alpha,\Gamma,\beta)\Rightarrow_f (\alpha_1,\Gamma_1,\beta_1)$.  If $(\alpha_n,\Gamma_n,\beta_n)$ is obtained from $(\alpha,\Gamma ,\beta)$ by a sequence of full $\mathscr{P}$-expansions then we denote this by   $(\alpha,\Gamma ,\beta)\Rightarrow^*_f(\alpha_n,\Gamma_n,\beta_n)$.

We now expand the notion of Sch\"{u}tzenberger graph to the Sch\"{u}tzenberger complexes. The Sch\"{u}tzenberger complexes were first defined by Steinberg in \cite{SB}. Later in \cite{SL}, Steven Linblad made a small modification in Steinberg's  definition of Sch\"{u}tzenberger complexes. In this paper, we are using  Linblad's definition of Sch\"{u}tzenberger complexes.   Let $M=\langle X|R\rangle $ be an inverse monoid and $m\in M$. The \textit{Sch\"{u}utzenberger complex} $SC(m)$ for $m\in M$ is defined as follows:

(1) The $1$-skeleton  of $SC(m)$ is the Sch\"{u}tzenberger graph $S\Gamma(m)$.

(2) For each relation $(r,s) \in R$ and vertex $v$, if $r$ and $s$ can be read at $v$, then there is a face with boundary given by the pair of paths labeled by $r$ and $s$ starting from $v$.

In similar manner, Stephen's approximate graphs can be viewed as approximate complexes by sewing on a face each time an elementary expansion is performed, and identifying faces if a determination results in their entire boundaries being identified.

\section{The word problem for Adian semigroups, Adian inverse semigroups and Adian groups}

The following theorem was first proved by Adian in \cite{AD} for finite presentations. Later, it was generalized by Remmers to any Adian presentation, in \cite{RM}, by using a geometric approach.

\begin{theorem}\label{embedding}
An Adian semigroup $Sg\langle X|R\rangle$ embeds in the corresponding Adian group $Gp\langle X|R\rangle$.
\end{theorem}

From the embedding in Theorem \ref{embedding}, we can derive the fact that every Adian semigroup embeds in an Adian inverse semigroup, as proved in the following proposition.
 \begin{proposition} \label{embedding 1}
 An Adian semigroup $S=Sg\langle X|R\rangle $ embeds in the Adian inverse semigroup $M=Inv\langle X|R\rangle $.
 \end{proposition}
 \begin{proof}
 Let $\theta :S \to M$ be the natural homomorphism and $\phi : S=\langle X|R\rangle \to G=Gp\langle X|R\rangle$ be the natural homomorphism. $\phi$ is an injective homomorphism  by \ref{embedding}. Note that $\phi=\sigma \circ \theta$. Since $\phi=\sigma \circ \theta$ and $\phi$ is injective, then $\theta$ must be injective.

 \end{proof}

 \begin{conjecture}
 (Adian, 1976) The word problem for Adian semigroups is decidable.
 \end{conjecture}

 \begin{remark}
 The word problem for one relation Adian semigroups is decidable. This is because Magnus \cite{Magnus} proved that the word problem for one relator groups is decidable and by Theorem \ref{embedding}, a one relation Adian semigroup embeds in the corresponding one relator Adian group.
 \end{remark}

 \begin{proposition}\label{WP}
 The word problem for an Adian semigroup $S=Sg\langle X|R\rangle $ and an Adian group $G=Gp\langle X|R\rangle $ is decidable, if:\\
 1. the Adian inverse semigroup $M=Inv\langle X|R \rangle$ is $E-$unitary and \\
 2. the word problem for the Adian inverse semigroup $M$ is decidable.
 \end{proposition}
 \begin{proof}
 It immediately follows from (1) that, for any word $u\in (X\cup X^{-1})^*$,  $u\underset{G}{=}1$ if and only if $u$ is an idempotent in $M$. So, if (2) holds, then we can check whether $u$ is an idempotent or not by checking the equality of words $u\underset{M}{=}u^2$.

 If (2) holds, then by Proposition \ref{embedding 1}, $S$ embeds in $M$ and so the word problem for $S$ is also decidable.

 \end{proof}

The following theorem proved in \cite{Eu} establishes the first part of the Proposition \ref{WP}.

\begin{theorem}\label{E-unitary} (Muhammad Inam, John Meakin, Robert Ruyle)
Adian inverse semigroups are $E-$unitary.
\end{theorem}

In order to solve the word problem for the Adian inverse semigroups we prove the following theorem which enables us to solve the word problem for some classes of Adian inverse semigroups.

\begin{theorem}\label{finite complexes}
Let $M=Inv\langle X|R\rangle$ be a finitely presented Adian inverse semigroup. Then the Sch\"{u}tzenberger complex of $w$ is finite for all words $w\in (X\cup X^{-1})^*$ if and only if  the Sch\"{u}tzenberger complex of $w'$ is finite for all positive words $w'\in X^+$.
\end{theorem}

\begin{remark}
We observe that if the word problem for a finitely presented Adian Inverse semigroup $Inv\langle X|R\rangle$ is decidable then the word problem for the corresponding Adian semigroup $Sg\langle X|R\rangle $  and the corresponding Adian group $Gp\langle X|R\rangle$ is also decidable. Because, if the word problem for an Adian inverse semigroup $Inv\langle X|R\rangle$ is decidable, then the corresponding Adian semigroup embeds into the Adian inverse semigroup by Proposition \ref{embedding 1}. So the decidability of the word problem for Adian semigroup follows from the decidability of the word problem of Adian inverse semigroup.  We know by Theorem \ref{E-unitary} that Adian inverse semigroups are $E$-unitary. So, if we want to check whether an element $w\in (X\cup X^{-1})^*$ is equal to the identity element of an Adian group $G=Gp\langle X|R\rangle$ in $G$, then we just need to check that whether $w$ is an idempotent in the corresponding Adian inverse semigroup $Inv\langle X|R\rangle$. Which can be checked immediately if the word problem for the Adian inverse semigroup is decidable.
\end{remark}

\section{Main Theorem}

 Let $Inv\langle X|R\rangle$ be an inverse semigroup. Then for any word $w\in (X\cup X^{-1})^*$, the sequence of approximate graphs $\{(\alpha_n,\Gamma_n(w),\beta_n)|n\in\mathbb{N}\}$ obtained by full $\mathscr{P}$-expansion over the presentation $\langle X|R\rangle$, converges to the Sch\"{u}tzenberger graph of $w$ over the presentation $\langle X|R\rangle$.  There exist graph homomorphisms, $\psi_n: \Gamma_n(w)\to \Gamma_{n+1}(w)$,   such that $\psi_n(\alpha_n)=\alpha_{n+1}$ and $\psi_n(\beta_n)=\beta_{n+1}$, for all $n\in \mathbb{N}$. If we attach to $\Gamma_n(w)$ 2-cells corresponding to the relations in the obvious way, we obtain an approximate complex of $SC(w)$. We use the same notation, so that  $\{(\alpha_n,\Gamma_n(w),\beta_n)|n\in\mathbb{N}\}$ becomes a sequence of approximate complexes that converges to $SC(w)$. We call a 2-cell to be an \textit{$n$-th generation 2-cell} if it occurs in $(\alpha_n,\Gamma_n(w),\beta_n)\setminus (\psi_{n-1}(\alpha_{n-1}),\psi_{n-1}(\Gamma_{n-1}(w)),\psi_{n-1}(\beta_{n-1}))$, for all $n\in \mathbb{N}$. The following lemma is due to Steinberg, who gives two slightly deferent proofs, in \cite{SB} and \cite{SB1}. The proof we present here is very similar to the proof in \cite{SB1}.

 \begin{lemma}\label{simply connected}
 Let $M=Inv\langle X|R\rangle$ be an inverse semigroup and $w\in (X\cup X^{-1})^*$. Then the Sch\"{u}tzenberger complex of $w$, $SC(w)$, is simply connected.

  \end{lemma}

 \begin{proof}
 We use induction on the sequence of finite approximate complexes obtained by full $\mathscr{P}$-expansion of $w$. Since $(\alpha_0,\Gamma_0(w),\beta_0)$ is just a tree, it is simply connected.

We assume that the the finite approximate complex $(\alpha_{k-1},\Gamma_{k-1}(w),\beta_{k-1})$ is simply connected and show that $(\alpha_{k},\Gamma_{k}(w),\beta_{k})$ is simply connected.

Each time we perform an elementary $\mathscr{P}$-expansion on $(\alpha_{k-1},\Gamma_{k-1}(w),\beta_{k-1} )$ we sew on a relation $(u,v)\in R$ along with a 2-cell bounded by the path labeled by $uv^{-1}$. We are exactly attaching a simply connected space along a continuous path, the result of which is again simply connected. Thus by induction, all complexes in the sequence of approximate complexes $\{(\alpha_n, \Gamma_n(w),\\ \beta_n):n\in \mathbb{N}\}$  are simply connected. It follows that the limit of $SC(w)$ of this sequence is simply connected.

\end{proof}

\begin{lemma} \label{no cycles 1}
Let $Inv\langle X|R\rangle$ be an Adian inverse semigroup and $w\in (X\cup X^{-1})^*$. Then the Sch\"{u}tzenberger complex of $w$ contains no directed cycles of 1-cells.
\end{lemma}
\begin{proof}
Meakin showed that, if $M=Inv\langle X|R\rangle $, is an $E$-unitary inverse semigroup, then for all words $u\in (X\cup X^{-1})^*$, the Sch\"{u}tzenberger graph of the word $u$ embeds into the Cayley graph of the maximal group homomorphic image of $M$,  $Gp\langle X|R\rangle$. Likewise, the Sch\"{u}tzenberger complex of $u$ embeds in the Cayley complex of the group $Gp\langle X|R\rangle$.  It has been proved in \cite{Eu} that Adian inverse semigroups are $E$-unitary. So $SC(w)$ embeds into the Cayley complex of the group $Gp\langle X|R\rangle$.

If $SC(w)$ contains a directed cycle then the Cayley complex of $Gp\langle X|R\rangle$ contains a  directed cycle as well.  We assume that that this directed cycle is labeled by a word $x$ for some $x\in X^+$. Then $x\underset{G}{=}1$ so there  exists a Van Kampen diagram with boundary labeled by $x$. But this contradicts Lemma 2(ii) of \cite{Eu} which shows that a Van Kampen diagram over an Adian presentation contains no directed cycles.
\end{proof}

The set of all edges (1-cells) of a graph (complex) whose tail vertex lies on a vertex labeled by $v$ is denoted $Star^o(v)$ and the set of all edges(1-cells) whose tip lies at a vertex $v$ is denoted by $Star^i(v)$.

\begin{lemma}\label{path}
 Let $M=Inv\langle X,R\rangle$ be an Adian inverse semigroup, $w\in X^+$, $v$ a vertex (0-cell) of an approximate complex $(\alpha_n, \Gamma_{n}(w),\beta_n)$ and let $a,b\in X$ label two distinct edges of $Star^o_{\Gamma_{n+1}(w)}(\psi_n(v))$. Then there exists a path in $LG\langle X|R\rangle$ connecting $a$ and $b$.
\end{lemma}
A dual statement to the above Lemma also holds for $Star^i_{\Gamma_{n+1}(w)}(\psi_n(v))$.
\begin{proof}
We use induction on $n$ to prove the above statement. If $n=0$, then there are the following two cases to consider.

\textit{Case 1.} Let $a\in Star^o_{\Gamma_0(w)}(v)$ and $b\in Star^o_{\Gamma_1(w)}(\psi_0(v))\setminus Star^o_{\psi_0(\Gamma_0(w))} (\psi_0(v\\ ))$. Since $w\in X^+$,  $(\alpha_0,\Gamma_0(w),\beta_0)$ is just a linear automaton with all of its edges directed towards the vertex labeled by $\beta_0$. If there exists an $R$-word  $ar$ (where $a\in X$ and $r\in X^*$) that labels a path of $\Gamma_0(w)$ from the vertex $v$ to a vertex $v'$, we can sew on a new path labeled by the other side of the relation from the vertex $v$ to the vertex $v'$. Note that this is the only way we can add new edges in $Star^o_{\Gamma_0(w)}(v)$, because if we attach a new path to the linear automaton $(\alpha_0,\Gamma_0(w),\beta_0)$ from a vertex $u$ to a vertex $u'$ and some vertex of this new path gets identified with the vertex $v$ of the linear automaton, then there will be a new edge in $Star^o_{\Gamma_0(w)}(v)$. But this is impossible unless $u=v$ because $\langle X|R\rangle$ is an Adian presentation, so the first letters of both the $R$-words in a relation are different from each other.

  As a consequence of attaching a new path starting from the vertex $v$ , $Star^o_{\Gamma_1(w)}(v)$ contains more than one element.  So, we sew on a path labeled by the other side of the relation $bs$ (where $b\in X$ and $s\in X^*$) from the vertex $v$ to the vertex $v'$, for some $(ar,bs)\in R$. So, there exists an edge between $a$ and $b$ in $LG\langle X|R\rangle$.

\textit{Case 2.} Suppose $a,b\in Star^o_{\Gamma_1(w)}(\psi_0(v))\setminus Star^o_{\psi_0(\Gamma_0(w))}(\psi_0(v))$. Then, since $(\alpha_0,\Gamma_0(w),\beta_0)$ is a linear automaton, there exist $R$-words $cs_1$ and $cs_2$ (where $c\in X$ and $s_1,s_2\in X^*$) labeling two overlapping segments of $\Gamma_0(w)$ starting from the vertex labeled by $v$. This allows us to sew on new paths labeled by the other sides of the relations $ar_1$ and $br_2$ (where $b\in X$ and $r_1,r_2\in X^*$), respectively, both starting from the vertex labeled by $v$. Since, $(ar_1,cs_1),(br_2,cs_2)\in R$, thus there exists a path between $a$ and $b$ in $LG\langle X|R\rangle$.

Now we assume that the above statement is true for $n=k$ and we prove the above statement for $n=k+1$.

If both the edges labeled by $a$ and $b$ belong to the set  $Star^o_{\Gamma_k(w)}(v)$, then by the induction hypothesis, there exists a path in $LG\langle X|R\rangle$ connecting $a$ and $b$. So, there is nothing to prove in this case.

Note that $\langle X|R\rangle$ is an  Adian presentation and therefore $SC(w)$ contains no positively or negatively labeled directed cycles by Lemma \ref{no cycles 1}. So, if $a\in Star^o_{\Gamma_k(w)}(v)$ and $b\in Star^o_{\Gamma_{k+1}(w)}(\psi_k(v))\setminus Star^o_{\psi_k(\Gamma_k(w))}(\psi_k(v))$, then either there exists a relation $(br,as)\in R$ (where $a,b\in X$ and $r,s\in  X^*$), where $as$ labels a path beginning at vertex $v$ in the birooted inverse word graph $(\alpha_k,\Gamma_k(w),\beta_k)$, or it occurs as consequence of sewing on a path labeled by an $R$-word starting from a vertex $v'$ (where the vertex labeled by $v'$  lies before the vertex labeled by $v$ on a positively labeled path from $\alpha_k$ to $\beta_k$), of $\Gamma_k(w)$ and then folding edges with the same label and same initial vertices.

The latter case is impossible, because if we read an $R$-word $ct$ (where $c\in X$ and $t\in X^*$) starting from the vertex labeled by $v'$ in the finite approximate complex $(\alpha_k,\Gamma_k(w),\beta_k)$ and we sew on a path labeled by the other side of the same relation, $du$ (where $d\in X$ and $u\in X^*$) to obtain $(\alpha'_k,\Gamma'_k(w),\beta'_k)$, then all the vertices of $LG\langle X|R\rangle$ that are labeled by those letters which also label the edges of the set $Star^o_{\Gamma_k(w)}(v')$, are connected by a path in $LG\langle X|R\rangle$, by our induction hypothesis. Since the edge labeled $d$ gets identified with one of the pre-existing edges in the set  $Star^o_{\Gamma_k(w)}(v')$, therefore there exists a path between $c$ and $d$ in $LG\langle X|R\rangle $. But, we also have $(ct,du)\in R$, i.e., there exists an edge between $c$ and $d$. So, there exists a closed path in $LG\langle X|R\rangle$. This is a contradiction.

Let $Star^o_{\Gamma_k(w)}(v)=\{ a_1,a_2,...,a_m|a_i\in X$ for $1\leq i\leq m\}$,  where each $a_i$ labels an edge. Then there are the following two cases to consider:

\textit{Case 1.} If $a\in Star^o_{\Gamma_k(w)}(v)$ and  $b\in Star^o_{\Gamma_{k+1}(w)}(\psi_k(v))\setminus Star^o_{\psi_k(\Gamma_k(w))}(\psi_k(\\ v))$. Then $a= a_i$ for some $i$ and there exists an $R$-word $as$ (where $a\in X$ and $s\in X^*$), that labels a path starting from the vertex labeled by $v$, which allowed us to sew on a path labeled by the other side of the same relation, for some relation of the form $(br,as)\in R$ (where $b\in X$ and $r\in X^*$). So, there exists an edge between $a$ and $b$ in $LG\langle X|R\rangle $ and $a$ is connected to $a_j$ for all $1\leq i\neq j\leq m$ by induction hypothesis. Hence, $b$ is connected with each $a_j$ for $1\leq j\leq m$ by a path in $LG\langle X|R\rangle$.

\textit{Case 2.} If   $a, b\in Star^o_{\Gamma_{k+1}(w)}(\psi_k(v))\setminus Star^o_{\psi_k(\Gamma_k(w))}(\psi_k(v))$. Then there exist $R$-words of the form $a_is_1$ and $a_js_2$ (where $a_i,a_j\in X$ and $s_1,s_2\in X^*$) for some $i,j\in \{1,2,...,m\}$, that can be read in the finite approximate complex $(\alpha_k,\Gamma_k(w),\beta_k)$, starting from the vertex labeled by $v$, that allowed us to sew on positively labeled paths labeled by the other sides of relations of the form $(ar_1,a_is_1), (br_2,a_js_2)\in R$ (where $a,b\in X$ and $r_1,r_2\in X^*$). Hence there exist an edge between $a$ and $a_i$ and an edge between $b$ and $a_j$ in $LG\langle X|R\rangle$. So, if $i=j$ , then there exists a path between $a$ and $b$ in $LG\langle X|R\rangle$ and if $i\neq j$, then $a_i$ and $a_j$ are distinct edges and $a_i$ is connected with $a_j$ in the $LG\langle X|R\rangle$ by the induction hypothesis. Hence, $a$ and $b$ are connected by a path in $LG\langle X|R\rangle$.

 \end{proof}

The following Corollary of Lemma \ref{path} is not useful for the proof of the main theorem of this paper. However, it provides some information about the structure of Adian semigroups.

\begin{corollary} Let $S=Sg\langle X|R\rangle$ be an Adian semigroup. Then $S$ has no idempotent element.
\end{corollary}
\begin{proof}
We assume that $S$ contains an idempotent element $w$ for some $w\in X^+$. It follows from the Proposition \ref{embedding 1} that $S$ embeds in the corresponding inverse semigroup $M=Inv\langle X|R\rangle$. So $w\underset{M}{=}w^2$. It follows from the Theorem \ref{Stephen's thm} that $w^2\in L(w)$. Then there exists a least positive integer $n$, such that $w^2$ labels a path from the initial vertex $\alpha_n$ to the terminal vertex $\beta_n$  of the approximate complex $(\alpha_n,\Gamma_n(w),\beta_n)$. We assume that $a\in X$ is the prefix letter of $w$. Then by Lemma \ref{path} there exists a path from $a$ to $a$ in $LG\langle X|R\rangle$. So $LG\langle X|R\rangle$ contains a closed a path. This is a contradiction.

\end{proof}

 The following lemma shows that if $Inv\langle X|R\rangle$ is an Adian inverse semigroup, then the construction of the Sch\"{u}tzenberger complex of a positive word only involves the elementary $\mathscr{P}$-expansion process and no folding at all.

 \begin{lemma}\label{no folding}

Let $M=Inv\langle X,R\rangle$ be an Adian inverse semigroup and $w\in X^+$. Then no two edges fold together in the construction of the Sch\"{u}tzenberger complex of $w$.
\end{lemma}
\begin{proof}
We use induction on $n$ to show that no two edges fold together in the construction of each finite approximate complex of the sequence $\{(\alpha_n,\Gamma_n(w),\beta_n)|n\in \mathbb{N}\}$.

The above statement is true for $n=0$, because $w\in X^+$, so $(\alpha_0,\Gamma_0(w),\beta_0)$ is just a linear automaton with no two consecutive edges oppositely oriented. Hence, no two edges fold together in the construction of $(\alpha_0,\Gamma_0(w),\beta_0)$.

We assume that the above statement is true for $n=k$, i.e., no two edges fold together in the construction of $(\alpha_k,\Gamma_k(w),\beta_k)$. We show that the above statement is also true for $n=k+1$. .

We apply an elementary expansion on $(\alpha_{k},\Gamma_{k}(w),\beta_{k})$  to obtain $(\alpha'_k,\Gamma'_{k}(w),\\ \beta'_{k})$ and then perform folding in $(\alpha'_k,\Gamma'_{k}(w),\beta'_{k})$  to obtain $(\alpha_{k+1},\Gamma_{k+1}(w),\beta_{k+1}\\ )$  in the full $\mathscr{P}$-expansion. So, for any vertex $v$ of  $(\alpha_k,\Gamma_k(w),\beta_k)$, if we read an $R$-word labeling a path starting from $v$ to a vertex $v'$, then we sew on a new path labeled by the other side of the same relation from $v$ to $v'$ and then we perform folding if possible. In order to prove the above statement we just need to show that no two edges fold together in $Star^o_{\Gamma'_k(w)}(v)$. The case that no two edges fold together in $Star^i_{\Gamma'_k(w)}(v))$ is dual to the previous case.

Let $\{a_i\in X|0\leq i\leq l\}$ be the set of labels of edges in $Star^o_{\Gamma_k(w)}(v)$ and let $b_0,b_1,...,b_m$ be the labels of those edges which are in $Star^o_{\Gamma'_k(w)}(v)\setminus Star^o_{\Gamma_k(w)}(v)$. We claim that

1. $a_i \neq b_j$ for $0\leq i\leq l$ and $0\leq j\leq m$, and

 2. $b_i\neq b_j$ for $0\leq i\neq j\leq m$.

  To establish our first claim, we assume that $a_i=b_j$ for some $0\leq i\leq l$ and $0\leq j\leq m$. It has already been shown in the proof  Lemma \ref{path} that if $b_j\in Star^o_{\Gamma'_k(w)}(v)\setminus Star^o_{\Gamma_k(w)}(v)$ then there exists a relation of the form $(b_jr,cs)\in R$ (where $b_j,c\in X$ and $r,s\in X^*$) and $cs$ labels a path from the vertex $v$ to vertex $v'$. It follows that $c \in \{a_t\in X|0\leq t\leq l\}$. If $c=a_i$, then $(b_jr,cs)=(a_ir,a_is)\in R$. This contradicts the fact that  $\langle X|R\rangle$ is an Adian presentation.

  If $c=a_h$ for some $h$ such that $0\leq h\neq i\leq l$, then by Lemma \ref{path}, $a_i$ and $a_h$ are connected by a path $p$ in $LG\langle X|R\rangle$. We also have $(b_jr,cs)=(a_hr,a_is)\in R$. So, there exists an edge $e$ between $a_i$ and $a_h$ in $LG\langle X|R\rangle$. If $p$ and $e$ represent the same path in $LG\langle X|R\rangle$, then the path labeled by $b_jr$ from the vertex $v$ to the vertex $v'$ already exists in $(\alpha_{k},\Gamma_{k}(w),\beta_{k})$, but this contradicts our earlier assumption that the path labeled by $b_jr$ from $v$ to $v'$ did not exist in $(\alpha_{k},\Gamma_{k}(w),\beta_{k})$. So, $p$ and $e$ represents two different co-terminal paths.  Hence there exists a cycle in $LG\langle X|R\rangle$. This is a contradiction.

To establish our second claim, we assume that $b_i=b_j$ for some $0\leq i,j\leq m$ with $i\neq j$. Note that the paths starting from the vertex $v$ with initial edges labeled by $b_i$ and $b_j$ are consequences of relations of the form $(b_ir_i,a_is_i),(b_jr_j,a_js_j)\in R$ respectively. If $a_i=a_j$, then either both of these relations are same or they are different relations. If both of these relations are same then this case reduces to the previous case, which has already been discussed above in the proof of claim 1. If these relations are different from each other, then there exists a cycle in $LG\langle X|R\rangle$, which is a contradiction. So, $a_i\neq a_j$. But then by lemma \ref{path}, there exists a path connecting $a_i$ with $a_j$ in $LG\langle X|R\rangle$. Also, there exist edges from $a_i$ to $b_i$ and $a_j$ to $b_j$ in $LG\langle X|R\rangle$, corresponding to the relators $(b_ir_i,a_is_i)$ and $(b_jr_j,a_js_j)$. This implies that there is a cycle in $LG\langle X|R\rangle$, which is a contradiction.

\end{proof}

\begin{proposition}
Let $M=Inv\langle X|R\rangle$ be an Adian inverse semigroup and $w_1\in X^+$. Then:
\begin{enumerate}
\item[(i)] $\psi_n:(\alpha_n ,\Gamma_n(w_1),\beta_n)\to (\alpha_{n+1}, \Gamma_{n+1}(w_1),\beta_{n+1})$ is an embedding for all $n\in \mathbb{N}$.

\item[(ii)] $SC(w_1)$ has exactly one source vertex $\alpha$ and exactly one sink vertex $\beta$, where $(\alpha ,\Gamma(w_1),\beta)$ is the underlying birooted graph of $SC(w_1)$.

\item[(iii)] Every directed edge of $SC(w_1)$ can be extended to a positively labeled directed transversal from $\alpha$ to $\beta$.

\end{enumerate}
\end{proposition}
\begin{proof}

\begin{enumerate}
\item[(i).]  (i) follows immediately from Lemma \ref{no folding}. Since no foldings occur in the construction of $\Gamma_{n+1}(w)$ from $\Gamma_n(w)$, the images of two distinct vertices of the approximate complex $(\alpha_n ,\Gamma_n(w),\beta_n)$  remain distinct under the map $\psi_n$, for all $n\in \mathbb{N}$.

\item[(ii).] (ii) follows from the fact that $\psi_n(\alpha_n)=\alpha_{n+1}=\alpha$ and $\psi_n(\beta_n)=\beta_{n+1}=\beta$ for all $n\in \mathbb{N}$. We sew on a new positively labeled path  (labeled by one side of a relation) to an approximate complex $(\alpha ,\Gamma_n(w),\beta)$ only when we read a positively labeled segment of a path from $\alpha$ to $\beta$, labeled by the other side of the same relation. Each 2-cell is two sided and no folding occurs in the construction of $(\alpha, \Gamma_{n+1}(w),\beta)$ from $(\alpha_n, \Gamma_{n}(w),\beta_n)$. So $\alpha$ and $\beta$ remain the source and sink vertices of the approximate complex $(\alpha ,\Gamma_n(w),\beta)$ for all $n\in \mathbb{N}$. Furthermore, $\alpha$ and $\beta$ were distinct vertices of $(\alpha,\Gamma_0(w_1),\beta)$ and no two vertices get identified with each other in the construction of $SC(w_1)$, so $\alpha $ and $\beta$ remain distinct vertices of $SC(w_1)$.

\item[(iii).] Let $e_0$ be a directed edge of $SC(w_1)$. The complex $SC(w_1)$ contains only one source vertex $\alpha$ and one sink vertex $\beta$. So, the initial vertex of $e_0$ is either $\alpha$ or it is a terminal vertex of another edge $e_{-1}$. If the initial vertex of $e_0$ is the terminal vertex of an edge $e_{-1}$ then the initial vertex vertex of  $e_{-1}$ is either $\alpha$ or the terminal vertex of an edge $e_{-2}$. Since there are no directed cycles in $SC(w_1)$ (by Lemma \ref{no cycles 1}), we can find a sequence of edges $e_{-n},e_{-(n-1)},...,e_{-1},e_0$ that constitutes a positively labeled directed path from $\alpha$ to the terminal vertex of edge $e_0$.

Similarly, if the terminal vertex of $e_0$ is not $\beta$ then we can find an edge $e_1$ whose initial vertex is the terminal vertex of $e_0$. Continuing in this way, we can find a sequence of edges $e_0,e_1,...,e_m$ that constitutes a positively labeled directed path from the edge $e_0$ to $\beta$. Hence, the  sequence of edges $e_{-n},e_{-(n-1)},...,e_0,...,e_m$ constitutes a positively labeled directed transversal from $\alpha$ to $\beta$.
\end{enumerate}
\end{proof}

 The word problem for $S=Sg\langle X|R\rangle$ is the question of whether there is an algorithm which given any two words $u,v\in X^+$, will determine whether $u=v$ in $S$.

   For any two words $u,v\in X^+$, $u\underset{S}{=}v$ if and only if there exists a transition sequence from $u$ to $v$.

   \begin{center}
   $u\equiv w_0\to w_1\to...\to w_n\equiv v$; for some $n\geq 0$,

   \end{center}

where $w_{i-1}\to w_i$ represents that $w_i$ is obtained from $w_{i-1}$ by replacing one side of a relation $r$ (that happens to be a subword of $w_{i-1}$)  with the other side $s$ of the same relation, for some $(r,s)\in R$. The above transition sequence is called a \textit{regular derivation sequence of length $n$} for the pair $(u,v)$ over the presentation $Sg\langle X|R\rangle$.

 A \textit{semigroup diagram} or \textit{$S$-diagram} over a semigroup presentation $Sg\langle X|R\rangle$ for a pair of positive words $(u,v)$  is a finite, planar cell complex $D \subseteq \mathbb{R}^2$, that satisfies the following properties:

   $\bullet$   The complex $D$ is connected and simply connected.

  $\bullet$ Each edge ($1$-cell) is directed and labeled by a letter of the alphabet $X$.

    $\bullet$ Each region (2-cell) of $D$  is labeled by the word $rs^{-1}$ for some defining relation $(r,s)\in R$.

  $\bullet$  There is a distinguished vertex $\alpha$ on the boundary of $D$ such that the boundary of $D$ starting at $\alpha$ is labeled by the word $uv^{-1}$. $\alpha$ is a source in $D$ (i.e. there is no edge in $D$ with terminal vertex $\alpha$).

  $\bullet$ There are no interior sources or sinks in $D$.


  In \cite{RM}, Remmers proved an analogue of Van Kampen's Lemma for semigroups to address the word problem for semigroups.

\begin{theorem}\label{Remmers thm}
Let $S=Sg\langle X|R\rangle$ be a semigroup and $u,v\in X^+$. Then there exists a  regular derivation sequence of length $n$ for the pair $(u,v)$ over the presentation $Sg\langle X|R\rangle$ if and only if there is an $S$-diagram over the presentation $Sg\langle X|R\rangle$ for the pair $(u,v)$ having exactly $n$ regions.

\end{theorem}

\begin{proposition}\label{equal}
Let $M=Inv\langle X|R\rangle$ be an Adian inverse semigroup and $w_1,w_2\in X^+$ such that $w_1\underset{M}{\leq} w_2$. Then:

\begin{enumerate}

\item[(i)]  There exists a (planar) $S$-diagram corresponding to the pair of words $(w_1,w_2)$ that embeds in $SC(w_1)$.

\item[(ii)] $w_1\underset{M}{=} w_2$.
\end{enumerate}
\end{proposition}

\begin{proof}

\begin{enumerate}

\item[(ii).] If $w_1\underset{M}{\leq} w_2$, then $w_2\in L(\alpha,\Gamma(w_1),\beta)$ (Stephen, \cite{SG}). So, $w_2$ labels a directed transversal from the vertex $\alpha$ to the vertex $\beta$. Since the transversals labeled by $w_1$ and $w_2$ are co-terminal and $SC(w_1)$ is simply connected by Lemma \ref{simply connected}, then the closed path labeled by the word $w_1w_2^{-1}$ is filled with finitely many 2-cells. Every 2-cell is two sided because $\langle X|R\rangle$ is an Adian presentation.  So, we can obtain a regular derivation sequence from the word $w_1$ to the word $w_2$ from the complex $SC(w_1)$, over the presentation $\langle X|R\rangle$ in the following way.

Geometrically, we push the transversal labeled by $w_1$ across all first generation 2-cells that were contained in the closed path labeled $w_1w_2^{-1}$ to obtain a new transversal labeled by $u_1\in X^+$. Combinatorially,  we have replaced some of the non overlapping $R$-words that were subwords of the word $w_1$ by the other side of the same relations. Then we push the transversal labeled by $u_1$ across all generation 2-cells that were contained in the closed path labeled by $w_1w_2^{-1}$ to obtain a new transversal labeled by $u_2\in X^+$. Again, we have just replaced some of the non overlapping $R$-words that were subwords of the word $u_1$. This process eventually terminates because the closed path labeled by $w_1w_2^{-1}$ contains only finitely many 2-cells. So, we obtain a regular derivation $w_1\to u_1\to u_2\to ...\to w_2$ over the presentation $\langle X|R\rangle$.  Hence there exists an $S$-diagram $\mathscr{S}$ corresponding to this derivation sequence.  Since no two edges fold together in $SC(w_1)$, therefore $\mathscr{S}$ embeds in $SC(w_1)$.

\item[(ii).] This follows immediately form $(i)$ and the Theorem \ref{Remmers thm}.
\end{enumerate}
\end{proof}

\begin{lemma}\label{bigon}

Let $M=Inv\langle X|R\rangle$ be an Adian inverse semigroup, let $w\in (X\cup X^{-1})^*$ and let $w_1,w_2\in X^+$ label two co-terminal paths in $SC(w)$. Then there exists an $S$-diagram corresponding to the pair of words $(w_1,w_2)$ that embeds in $SC(w)$.
\end{lemma}
\begin{proof}

 Since $M$ is $E$-unitary, $SC(w)$ embeds into the Cayley complex of the group $G=Gp\langle X|R\rangle$. So, the word $w_1w_2^{-1}$ labels a closed path in the Cayley complex. Hence $w_1w_2^{-1}\underset{G}{=}1$. So, $w_1\underset{G}{=}w_2$. It follows from Theorem \ref{embedding} that $w_1=w_2$ in $Sg\langle X|R\rangle$. So, there exists a regular derivation sequence $D$ from $w_1$ to $w_2$ over the presentation $S=Sg\langle X|R\rangle$. The semigroup $S$ embeds into the inverse semigroup $M$, by Proposition \ref{embedding 1}. So the regular derivation sequence sequence $D$ also holds in $M$. The complex $SC(w)$ is closed under elementary $\mathscr{P}$-expansion and folding and therefore all the 2-cells corresponding to this regular derivation sequence $D$ already exist in $SC(w)$ between the paths labeled by $w_1$ and $w_2$. Hence the $S$-diagram corresponding to the regular derivation sequence  $D$ embeds in $SC(w)$.

\end{proof}

In general for any finite inverse semigroup presentation $(X, R)$, and for any word $w\in (X\cup X^{-1})^*$, there are natural birooted graph morphisms for the sequence of approximate graphs given by Stephen's procedure for approximating the birooted Sch\"{u}tzenberger graph $(\alpha, \Gamma(w), \beta)$.
\\ $ (\alpha_1 ,\Gamma_1(w),\beta_1)\to \cdots \to(\alpha_n ,\Gamma_n(w),\beta_n)\to (\alpha_{n+1} ,\Gamma_{n+1}(w),\beta_{n+1})\to \cdots \to$ \\ $ (\alpha, \Gamma(w), \beta)$

It was established in Proposition \ref{equal} that if $M=Inv\langle X|R\rangle$ is an \textit{Adian inverse semigroup} and $w$ is a \textit{positive} word, that is $w\in X^+$, then all of the maps in the above sequence are actually embeddings. We may abuse the notation slightly in this case and for each approximate graph $\Gamma_n(w)$ we denote the initial and terminal vertices simply as $\alpha$ and $\beta$.
\[ (\alpha ,\Gamma_1(w),\beta)\hookrightarrow \cdots \hookrightarrow(\alpha ,\Gamma_n(w),\beta)\hookrightarrow (\alpha ,\Gamma_{n+1}(w),\beta)\hookrightarrow \cdots \hookrightarrow(\alpha, \Gamma(w), \beta)\]

Recall that for a positive word $w\in X^+$, a \textit{transversal} of an approximate Sch\"{u}tzenberger complex $(\alpha, \Gamma_k(w),\beta)$ is defined to be any positively labeled path in $\Gamma_k(w)$ from $\alpha$ to $\beta$. Before proving the main theorem of this section we introduce the following definition of \textit{$n$-th generation transversal} of a Sch\"{u}tzenberger complex.

\begin{definition}
An \textit{$n$-th generation transversal} of the Sch\"{u}tzenberger complex $(\alpha, \Gamma(w),\beta)$  is a positively labeled path from $\alpha$ to $\beta$ that can be read in the approximate complex $(\alpha,\Gamma_n(w),\beta)$ but cannot be read in $(\alpha,\Gamma_{n-1}(w),\beta)$, for some $n\in \mathbb{N}$.
\end{definition}

When studying certain problems involving an inverse monoid given by a finite presentation $(X, R)$, for example when considering the word problem for $M = Inv\langle X|R \rangle$, it is natural to first ask whether it might happen to be the case that all of the Sch\"{u}tzenberger graphs of $M$ are finite. The following theorem, our main theorem for this paper, shows that this question can be reduced to the question of whether or not all Sch\"{u}tzenberger graphs of \textit{positive} words are finite.

\begin{theorem}\label{finite complexes}
Let $M=Inv\langle X|R\rangle$ be a finitely presented Adian inverse semigroup. Then the Sch\"{u}tzenberger complex of $w$ is finite for all words $w\in (X\cup X^{-1})^*$ if and only if  the Sch\"{u}tzenberger complex of $w'$ is finite for all positive words $w'\in X^+$.
\end{theorem}

\textbf{Idea of the proof:} We assume that the Sch\"{u}tzenberger graph of every positive word is finite and we let $w$ be an arbitrary word, $w\in (X\cup X^{-1})^{*}$. We will use induction on the number of edges in the Munn tree $MT(w)$ to prove that, the Sch\"{u}tzenberger complex of $w$ is finite. The essential part of the proof involves realizing $SC(w)$ as the limit of a sequence of finite complexes, via the procedure of Stephen's $\mathscr{P}$-expansion. We begin with a finite inverse graph (complex) $S$ that is closed under $\mathscr{P}$-expansions relative to $\langle X| R\rangle$. To the complex $S$ we attach a single edge $e$ that is labeled by some letter, say $a\in X$. The resulting complex $S_1 = S\vee \{e\}$ will in general not be closed under $\mathscr{P}$-expansions. In a process similar to Stephen's full $\mathscr{P}$-expansion construction, we define a sequence of finite complexes $S_1, S_2, \dots$ that converges in the limit to $SC(w)$. Our theorem will be proved if we can show that this limit is in fact a finite complex. Equivalently, we must prove that the sequence $S_1, S_2, \dots$ stabilizes after finitely many steps at some $S_k$.

\begin{proof}
We assume that $SC(w')$ is finite for all positive words $w'\in X^+$,. Let $w$ be an arbitrary word, $w\in (X\cup X^{-1})^*$. We will show that $SC(w)$ is  finite, by applying induction on the number of edges in $MT(w)$.

For the base of our induction, we suppose that $MT(w)$ consists of only one edge, labeled say by $a\in X$. Then by our assumption about positive words, $SC(w) = SC(a)$ is finite.

For our induction hypothesis, we assume that $SC(w_0)$ is finite for all words $w_0\in (X\cup X^{-1})^*$, whose Munn tree consist of $k$ edges. Let $w\in (X\cup X^{-1})^*$ be a word such that $MT(w)$ consists of $k+1$ edges. We will show that $SC(w)$ is finite.

Let $\alpha$ be an extremal vertex of $MT(w)$ (i.e., a \textit{leaf} of the tree $MT(w)$), and let $e$ be the edge of $MT(w)$ that connects $\alpha$ to the remainder of the $MT(w)$. Assume that $e$ is a positively labeled edge with initial vertex $\alpha$ and terminal vertex $\beta$. The case when $e$ is negatively labeled  is dual. The sub-tree obtained by removing the edge $e$ consists of $k$ edges; we denote this sub-tree by $T$ for our future reference. Note that the tree $T$ is in fact the Munn tree of some word $z \in (X\cup X^{-1})^*$. That is, $T = MT(z)$ and there are $k$ edges in $MT(z)$.  By the induction hypothesis, the Sch\"{u}tzenberger complex $S$ generated by $T$, i.e., $ S = SC(z)$, is a finite complex. There exists a graph morphism $\phi:T\to S$, and so we may regard the vertex $\beta$ of $T$ as a vertex in the complex $S = SC(z)$.  We reattach the edge $e$ to the vertex $\beta$ of $S$ and denote the resulting finite complex by $S_1$, (see Figure \ref{fig301}.) The finite complex $S = SC(z)$ was obtained from the subtree $T$ by sewing on relations from $R$, and since we reattached the edge $e$ in $S_1 = S\vee \{e\}$, then naturally $S_1$ may be regarded as an approximation to $SC(w)$.  While the Sch\"{u}tzenberger complex $S = SC(z)$ is closed under $\mathscr{P}$-expansion, the complex $S_1$ is not necessarily closed under $\mathscr{P}$-expansion. In particular, it is possible that there may be one or more relations $(r, s)\in R$ such that $r$ labels a path in $S_1$ that begins at vertex $\alpha$ and the other side of the relation, $s$, is not read in $S_1$ at $\alpha$. It is clear that the closure of $S_1$ under $\mathscr{P}$-expansion over $\langle X| R\rangle$ is the Sch\"{u}tzenberger complex $SC(w)$.

\begin{figure}[h!]
\centering
\includegraphics[trim = 0mm 0mm 0mm 0mm, clip,width=1.8in]{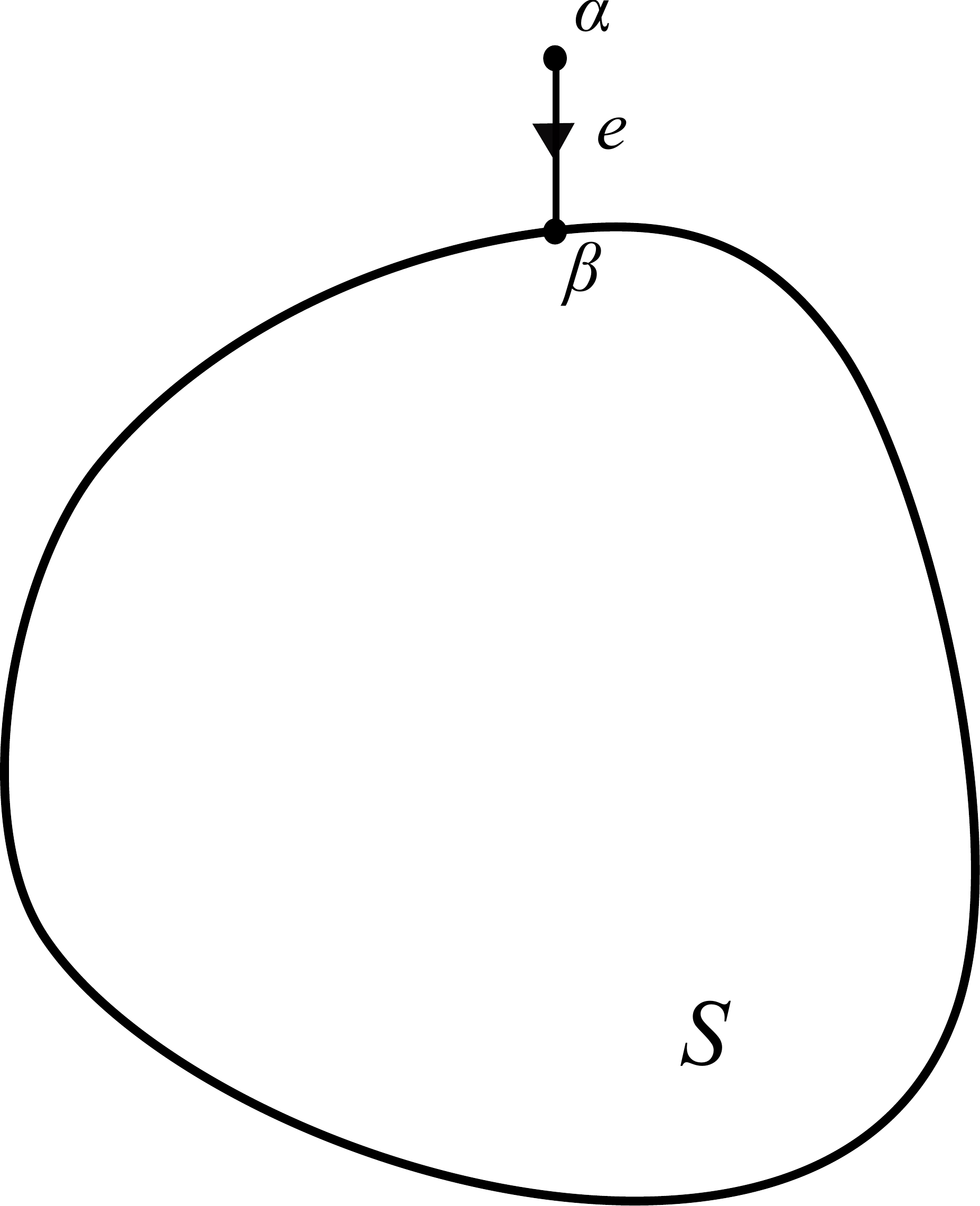}
\caption{$S_1=S\vee \{e\}$ }
\label{fig301}
\end{figure}

If the edge $e$ (labeled say, by $a$) gets immediately identified by folding to an $a$-labeled edge of the finite complex $S$, then we are done, because in that case we would have $SC(w)=S$, which was assumed to be a finite complex. So, we assume that the edge $e$ does not get immediately identified by folding with any of the edges in the finite complex $S$. We extend the edge $e$ to all possible maximal positively labeled paths in $S_1$. There are only finitely many maximal positively labeled paths in $S_1$ with initial edge $e$ because $S_1$ is a finite complex that has no positively labeled cycles (by Lemma \ref{no cycles 1}). We assume that  these paths are labeled by $w_1,w_2,...,w_n$, where $w_i\in X^+$ for $1\leq i\leq n$. Each such $w_i$ labels a path from $\alpha$ to some vertex $\beta_i$ of $S$. (See Figure \ref{fig302}).

\begin{figure}[h!]
\centering
\includegraphics[trim = 0mm 0mm 0mm 0mm, clip,width=2in]{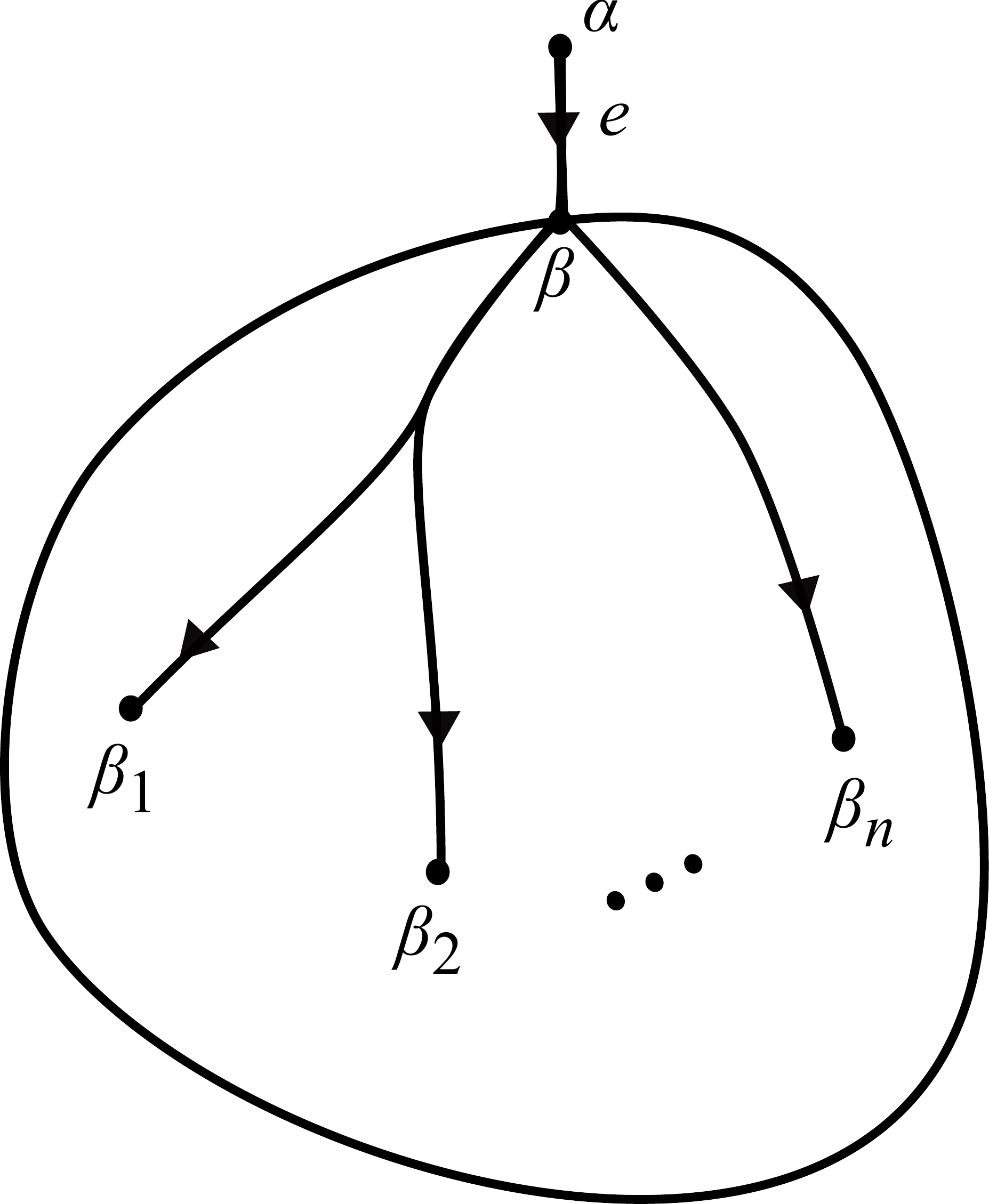}
\caption{$S_1$ with maximal positively labeled paths from $\alpha$ to $\beta_i$. }
\label{fig302}
\end{figure}

 In order to complete $S_1$ under elementary $\mathscr{P}$-expansion and folding, we first attach $SC(w_i)$ to the path labeled by $w_i$ in $S_1$, for all $i$, and denote the resulting finite complex by $S'_1$. (See Figure \ref{fig303}.) Each complex $SC(w_i)$ for $1\leq i\leq n$, is finite since each $w_i\in X^+$. Thus we obtain $S'_1$ by attaching to $S_1$ finitely many complexes, each of which is finite, and so $S'_1$ is finite.

\begin{figure}[h!]
\centering
\includegraphics[trim = 0mm 0mm 0mm 0mm, clip,width=2in]{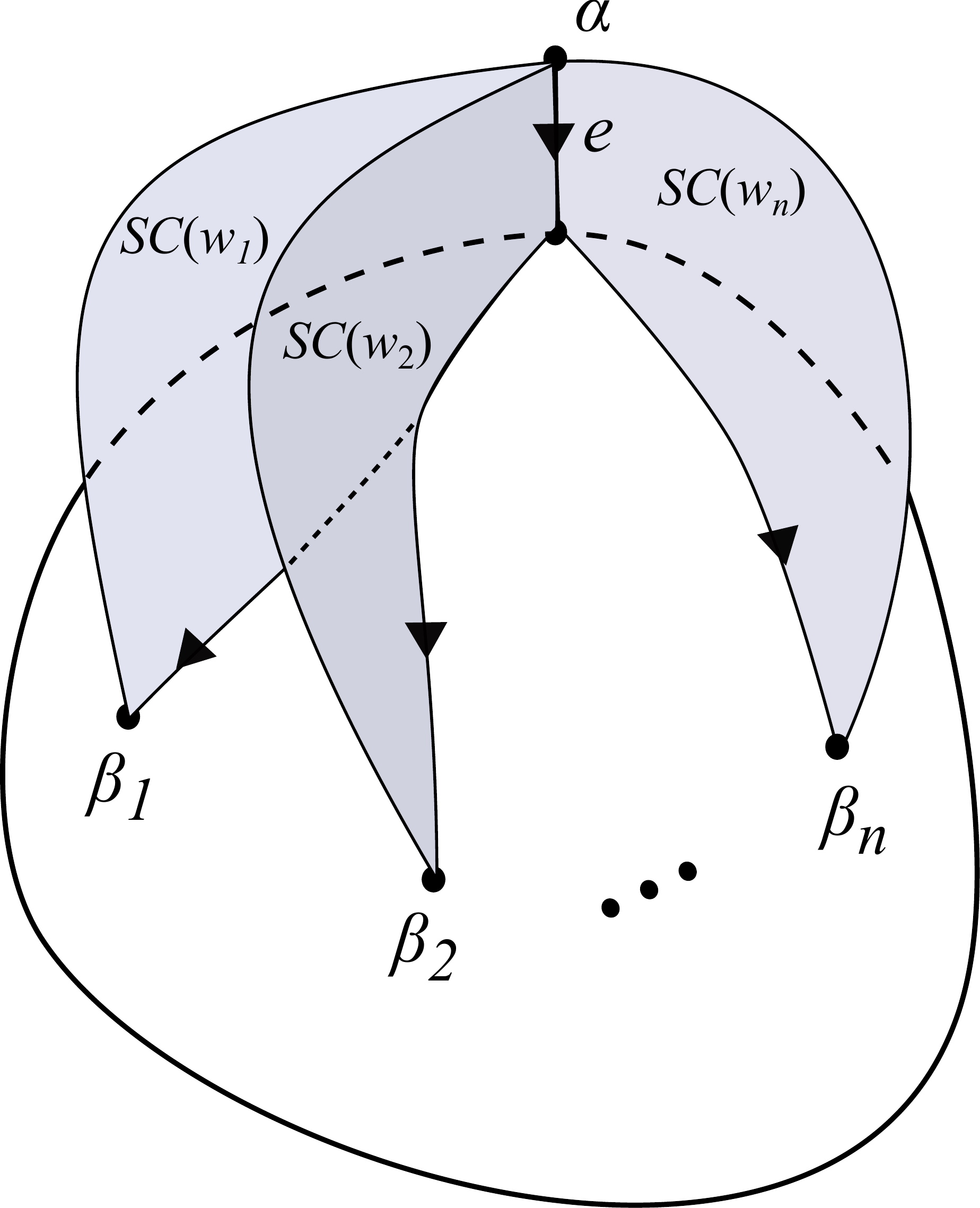}
\caption{$S_1'$ is obtained from $S_1$ by sewing on $SC(w_i)$ to each maximal positive path, with label $w_i$, that begins at $\alpha$ in $S_1$.}
\label{fig303}
\end{figure}

The complex $S'_1$ is not necessarily determinized. That is, as a consequence of attaching the complexes $SC(w_i)$, there may now be vertices along the paths in $S'_1$ labeled by $w_i$ at which there exist two or more edges labeled by the same letter and so we must perform foldings to obtain a determinized graph (complex). We denote the determinized form of $S'_1$ by $S_2$. Since $S'_1$ is finite, then so is its determinized quotient $S_2$. Note however, as consequence of folding $S'_1$ to $S_2$, that the complex $S_2$ may not be closed under $\mathscr{P}$-expansions.

 And so, we iterate the procedure. In general, the complex $S_k$ may not be closed under $\mathscr{P}$-expansion. That is, there may be a relation $(r, s)\in R$ and a path $p$ labeled by $r$ between some two vertices $v$ and $v'$ in $S_k$ such that the other side of the relation, $s$, does not label a path in $S_k$ between $v$ and $v'$. We refer to such a path $p$ in $S_k$ as an \textit{unsaturated} path.
 
 Note that in Stephen's procedure we would at this stage simply attach a path labeled by $s$ between $v$ and $v'$ as one step in the $\mathscr{P}$-expansion. In our setting it turns out, however, that we can ``speed up" Stephen's procedure. We will show that every such unsaturated path $p$ can be extended to a positively labeled path that begins at the vertex $\alpha$. Such a path can itself be extended to a maximal positively labeled path that begins at $\alpha$. So, instead of merely attaching a single cell along the unsaturated path $p$ that is labeled by $r$, we instead read the label, say $w_j$, of the maximal positively labeled path beginning at $\alpha$ that contains $p$ as a subpath and we attach the (finite) complex $SC(w_j)$ along this path. Thus in this one step we are attaching not only the one cell along the unsaturated path $p$, but also are attaching all cells that would arise from $\mathscr{P}$-expansions on the maximal path labeled by $w_j$.

 Our iterative procedure for constructing a sequence $\{S_n\}$ of complexes can be summarized as follows. Suppose that $S_k$ has been constructed. We first look for all positive words $w_i$ that label a maximal positive path in $S_k$ that starts at the vertex $\alpha$ and that does not label a path starting at $\alpha$ in $S'_{k-1}$. We obtain $S'_k$ by attaching $SC(w_i)$ to each such maximal positive path in $S_k$. Then we obtain $S_{k+1}$ by determinizing $S'_k$. Thus we obtain a sequence of finite complexes \{$S_n\}$ that has $SC(w)$ as its limit.

\[ \begin{array}{ccccc} S_1 \rightarrow \cdots  \rightarrow   S_k & \xrightarrow{\hspace*{3.5cm}} &  S'_k  &  \xrightarrow{\hspace*{2.8cm}} & S_{k+1}  \rightarrow  \cdots \\[-10pt]
& \parbox[t]{3.5cm}{\scriptsize For each maximal positive
path with label $w_i$ that
starts at $\alpha$ and does not
exist in $S'_{k-1}$, attach $SC(w_i).$} &&
\parbox{2.8cm}{\scriptsize Fold. (Determinize.)} & \\ &&&&
\end{array} \]

A question is whether we ever reach a graph $S_k$ in the procedure such that all maximal positive paths that can be read at $\alpha$ in $S_k$ can already be read in $S'_{k-1}$ at $\alpha$. If this happens, then the sequence of complexes will stabilize at $S_k$. Equivalently, we can ask whether we ever reach a complex $S_k$ such that every maximal positive path in $S_k$ is closed under $\mathscr{P}$-expansion. (We say that a maximal positive path $p$ is closed under $\mathscr{P}$-expansion if for every relation $(r, s)\in R$, and for every subpath of $p$ labeled by $r$ between vertices $v$ and $v'$, then the other side of the relation, $s$, already labels a path between $v$ and $v'$ in $S_k$).

To answer this question, we analyze how the process of folding $S'_k$ to $S_{k+1}$ affects positively labeled paths that already exist in $S'_k$ and how new positively labeled paths may be created in $S_{k+1}$ as a consequence of the folding process. We first note, for example, that when we fold $S'_1$ to $S_2$ that each of the attached complexes $SC(w_i)$ that we attached to $S_1$ will be embedded in $S_2$ after the folding process. This follows from a fact (see Stephen, \cite{SG}) about  $E$-unitary semigroups: If $M$ is $E$-unitary and $w$ is any word, and a word $w_i$ can be read along some path of the Sch\"{u}tzenberger graph $\Gamma(w)$, then the entire Sch\"{u}tzenberger graph $\Gamma(w_i)$ will occur as an embedded subgraph of $\Gamma(w)$ along that path. Since the theorem we are proving assumes that the semigroup $M$ has an Adian presentation, we know from the main theorem of \cite{Eu} that $M$ is $E$-unitary. Thus, each $SC(w_i)$ embeds in $S_2$. Likewise, the original graph $S$ is actually the Sch\"{u}tzenberger graph $\Gamma(z)$, where the word $z$ labels a path in $S_1$, and so we know that the original graph $S$ also remains embedded as a subgraph of $S_2$.  In other words, no two vertices of any one graph $SC(w_i)$ will become identified with each other and no two vertices of the original graph $S$ will become identified with each other in the folding process that takes $S'_1$ to $S_2$. Two vertices of $S'_1$ will become identified in the folding process only if one of the vertices belongs to the original $S$ and the other vertex belongs to one of the attached $SC(w_i)$,  or if the two vertices belong to $SC(w_i)$ and $SC(w_j)$, with $i\neq j$. Everything that we just said about the process of folding $S'_1$ to $S_2$ holds as well for the process of folding $S'_k$ to $S_{k+1}$. The original complex $S$ and each complex $SC(w_i)$, attached at any step of the iteration, will be embedded as subcomplexes of $S_{k+1}$. The interest is in what new paths may be formed as a result of folding $S'_k$ to $S_{k+1}$.

\textbf{Claim:} \textit{ Suppose that $SC(w_i)$ is one of the complexes that was attached to $S_k$ to form $S'_k$. Suppose, in the process of folding $S'_k \rightarrow S_{k+1}$, that a vertex $\gamma$ of $SC(w_i)$ gets identified, as a consequence of folding, with a vertex $\gamma'$ of the original complex $S$. Then we claim that every positively labeled path in $SC(w_i)$ from $\gamma$ to $\beta_i$ will get identified by the folding process with a path in $S$ from the vertex $\gamma'$ to $\beta_i$.  Thus, every maximal positive path $p$ that begins at the vertex $\alpha$ in $S_{k+1}$ and cannot be read beginning at the vertex $\alpha$ in $S'_k$  will factor uniquely as $p = p_1p_2$, where $p_1$ is a path in some $SC(w_i)$s and $p_2$ is a path in the original complex $S$.}

\begin{proof}[Proof of Claim:]
Since $\gamma$ gets identified with $\gamma'$ through folding, we know (Stephen, \cite{SG}) that in $S'_k$ there is a path from $\gamma$ to $\gamma'$ labeled by a Dyck word $d$ that we may assume is of the form $d=ss^{-1}$, where $s\in (X\cup X^{-1})^*$. Further, we may assume that $s$ is a reduced word. There must be a vertex $\delta$ that lies on the path labeled by $w_i$ (the intersection of $SC(w_i)$ and $S$ in $S'_k$) so that $s$ labels a path in $SC(w_i)$ from $\gamma$ to $\delta$ and $s$ also labels a path in $S$ from $\gamma'$ to $\delta$. (See Figures \ref{fig52} \& \ref{fig53}.)
To prove the above claim we assume that $r_2$ labels an arbitrary positive path in $SC(w_i)$ from $\gamma$ to $\beta_i$. We need to prove that this path gets identified by folding with a path in $S$  that is labeled by $r_2$ from $\gamma'$ to $\beta_i$. The path in $SC(w_i)$ that is labeled by $r_2$ can be extended to an $n$th generation transversal $t$ of $SC(w_i)$, for some $n$. To complete our proof, we apply induction on the generation number, $n$, of the transversal of $t$ of $SC(w_i)$.

Suppose that the path labeled by $r_2$ lies on a 1st generation transversal $t$ of $SC(w_i)$. We assume that $t\equiv r_1r_2$ where $r_1$ labels the sub-path of $t$ from $\alpha$ to $\gamma$ and $r_2$ labels the sub-path of $t$ from $\gamma$ to $\beta_i$. Since $\gamma$ lies on a 1st generation transversal of $SC(w_i)$, we conclude that in the Dyck word path $ss^{-1}$ that goes from $\gamma$ to $\delta$ to $\gamma'$, it must be that either $s$ is a positive word or $s^{-1}$ is a positive word. We examine the following three cases.

\textit{Case 1.} (See Figure \ref{fig51}.) Suppose that the vertex $\delta$ is actually $\beta_i$, the terminal vertex of the path labeled by $w_i$, and suppose the Dyck word labeling the path from $\gamma$ to $\gamma'$ is the word $r_2r_2^{-1}$. In this case, the word $r_2$ labels the subpath of the transversal $t$ of $SC(w_i)$ from $\gamma$ to $\beta_i$ and the word $r_2$ also labels a path from $\gamma'$ to $\beta_i$.  These two paths, both labeled by $r_2$, which meet at the vertex $\beta_i$, will fold together so that the path labeled by $r_2$ in $SC(w_i)$ from $\gamma$ to $\beta_i$ gets identified with the path in $S$ from $\gamma'$ to $\beta_i$.  So in this case the statement of the claim obviously holds.

\begin{figure}[h!]
\centering
\includegraphics[trim = 0mm 0mm 0mm 0mm, clip,width=2.5in]{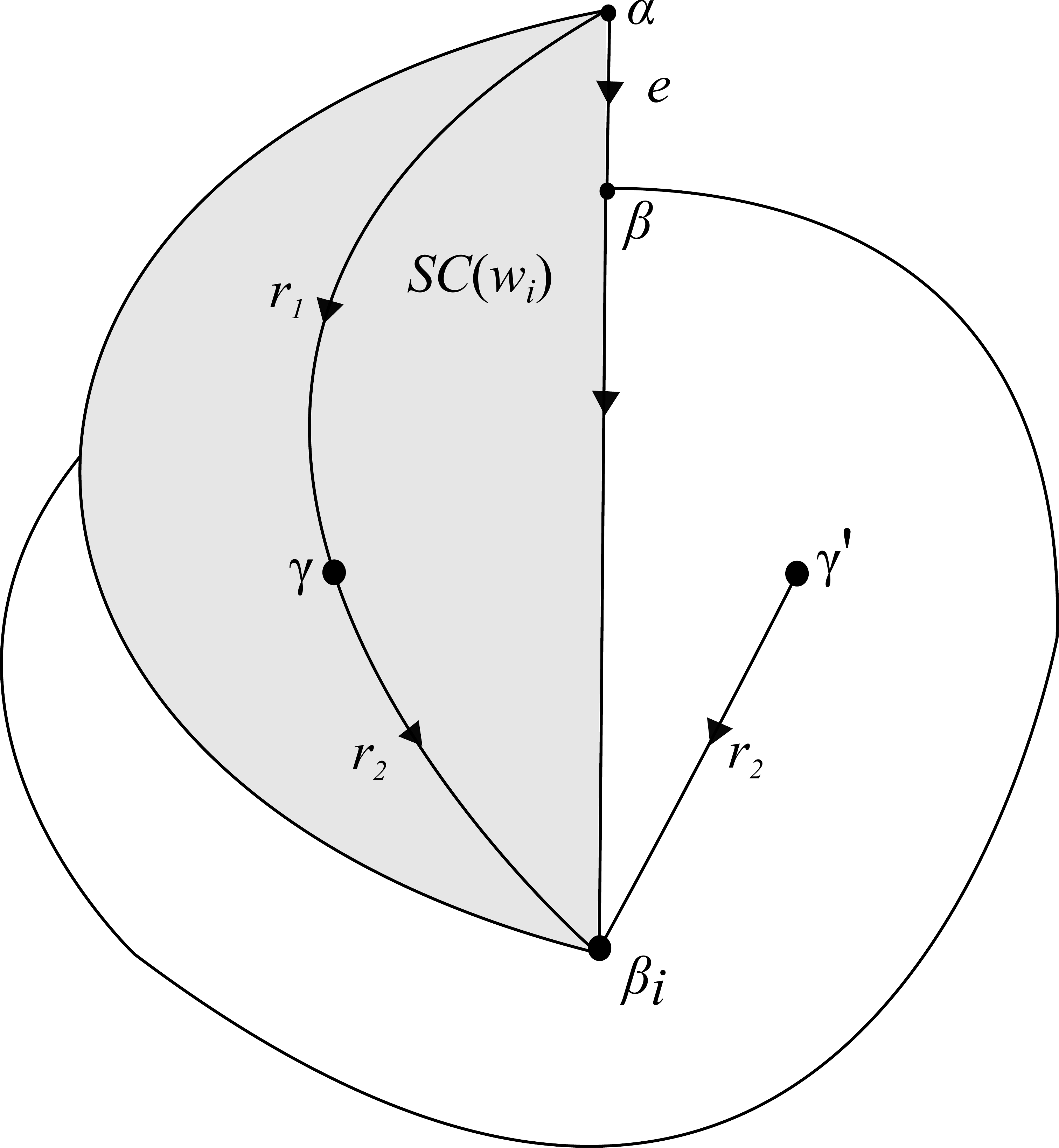}
\caption{Case 1 for the first generation transversals of the proof of the Claim }
\label{fig51}
\end{figure}

\textit{Case 2.} (See Figure \ref{fig52}.) We assume that the Dyck word path from $\gamma$ to $\delta$ to $\gamma'$ is labeled by $r_3r_3^{-1}$, where $r_3$ is a positive word, $r_3\in X^+$. So, the path in $SC(w_i)$ from $\gamma$ to $\delta$ is labeled by $r_3$ and the path in $S$ from $\gamma'$ to $\delta$ is labeled by $r_3$. The oppositely oriented paths labeled by $r_3$, which meet at $\delta$, become identified with each other through folding. We assume that the sub-path from $\delta$ to $\beta_i$, of the maximal path labeled by $w_i$, is labeled by $r_4\in X^+$. Obviously, the path labeled by $r_4$ from $\delta$ to $\beta_i$ is in $SC(w_i)$ and the path labeled by $r_3$ from $\gamma$ to $\delta$ is also in $SC(w_i)$. Hence the path labeled by $r_3r_4$ from $\gamma$ to $\beta_i$ is in $SC(w_i)$. The positive words $r_2$ and $r_3r_4$ label two co-terminal paths in $SC(w_i)$. So, by Lemma \ref{bigon} an $S$-diagram corresponding to the pair of words $(r_2, r_3r_4)$ embeds in $SC(w_i)$. This $S$-diagram also embeds in $S$, because $S$ contains a path labeled by one side of this $S$-diagram, (namely, the path from $\gamma'$ to $\beta_i$ in $S$ that is labeled by $r_3r_4$), and $S$ is closed under elementary $\mathscr{P}$-expansion. So, the two $S$-diagrams corresponding to the pair of words $(r_2,r_3r_4)$ get identified with each other and our claim holds in this case. That is, the path in $SC(w_i)$ labeled by $r_2$ from $\gamma$ to $\beta_i$ gets identified with a path in $S$ from $\gamma'$ to $\beta_i$.

\begin{figure}[h!]
\centering
\includegraphics[trim = 0mm 0mm 0mm 0mm, clip,width=2.5in]{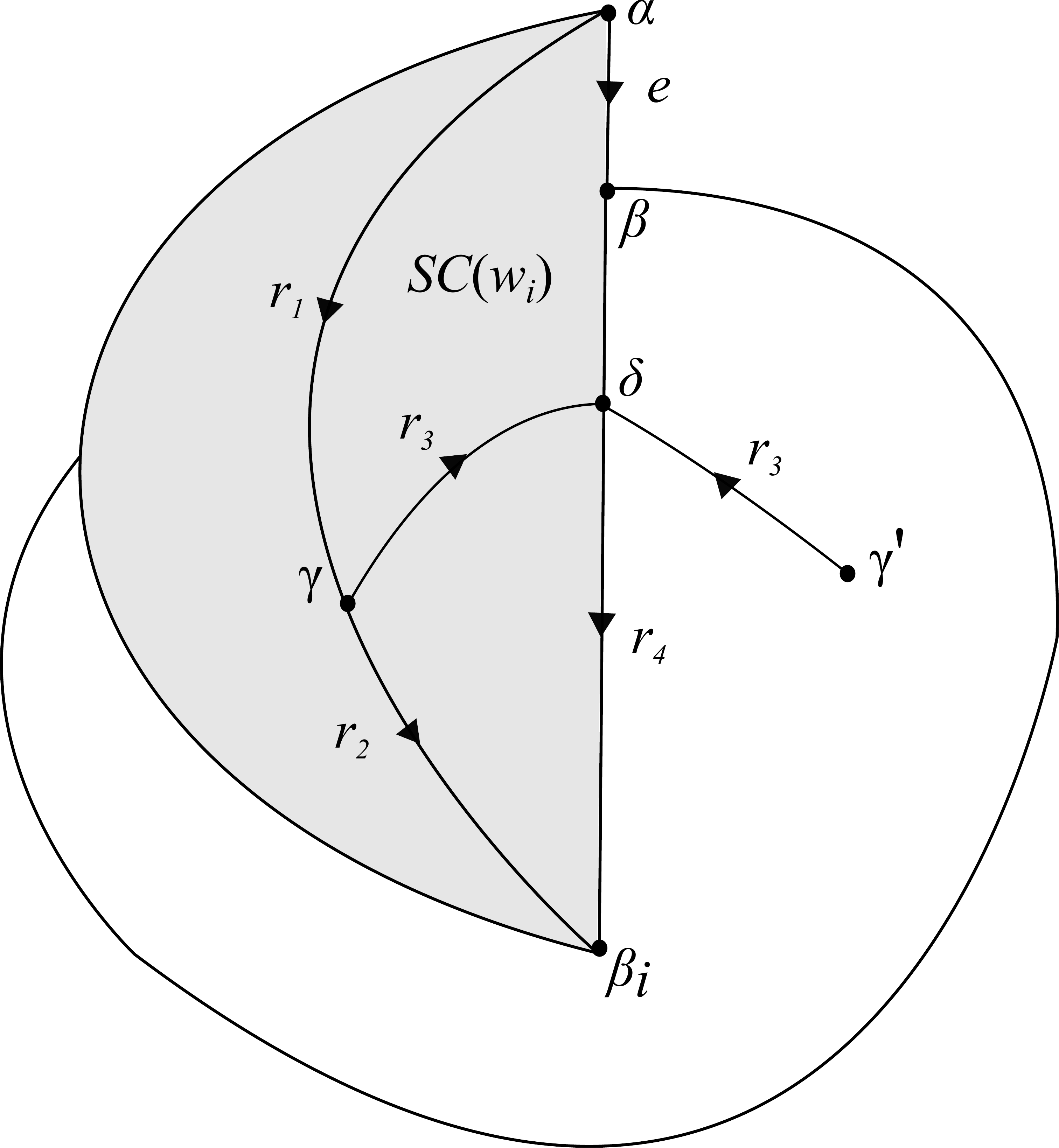}
\caption{ Case 2 for the first generation transversals of the proof of the Claim}
\label{fig52}
\end{figure}

\textit{Case 3.} (See Figure \ref{fig53}.) We assume that the Dyck word path from $\gamma$ to $\delta$ to $\gamma'$ is labeled by $r_3^{-1}r_3$, where $r_3$ is a positive word. So, the path in $SC(w_i)$ from $\delta$ to $\gamma$ is labeled by $r_3$ and the path in $S$ from $\delta$ to $\gamma'$ is labeled by $r_3$. The oppositely oriented paths labeled by $r_3$, which meet at $\delta$, become identified with each other through folding.  Now we have the path in $SC(w_i)$ labeled by $r_3r_2$  from $\delta$ to $\beta_i$ in $SC(w_i)$. Again, we let the positive word $r_4$ be the label of the sub-path of the maximal path labeled by $w_i$ from $\delta$ to $\beta_i$. The words $r_3r_2$ and $r_4$ label two co-terminal paths in $SC(w_i)$. So, by Lemma \ref{bigon} an $S$-diagram corresponding to the pair of words $(r_3r_2,r_4)$ embeds in $SC(w_i)$. This $S$-diagram also embeds in $S$, because $S$ contains a path labeled by one side of this $S$-diagram (namely, the path labeled by $r_4$) and $S$ is closed under elementary $\mathscr{P}$-expansion. Hence these two $S$-diagrams get identified with each other through folding. In particular, the path in $SC(w_i)$ labeled by $r_2$ from $\gamma$ to $\beta_i$ gets identified with a path in $S$ from $\gamma'$ to $\beta_i$, and so our claim follows in this case as well. This concludes the base case of the inductive proof of the claim.

\begin{figure}[h!]
\centering
\includegraphics[trim = 0mm 0mm 0mm 0mm, clip,width=2.5in]{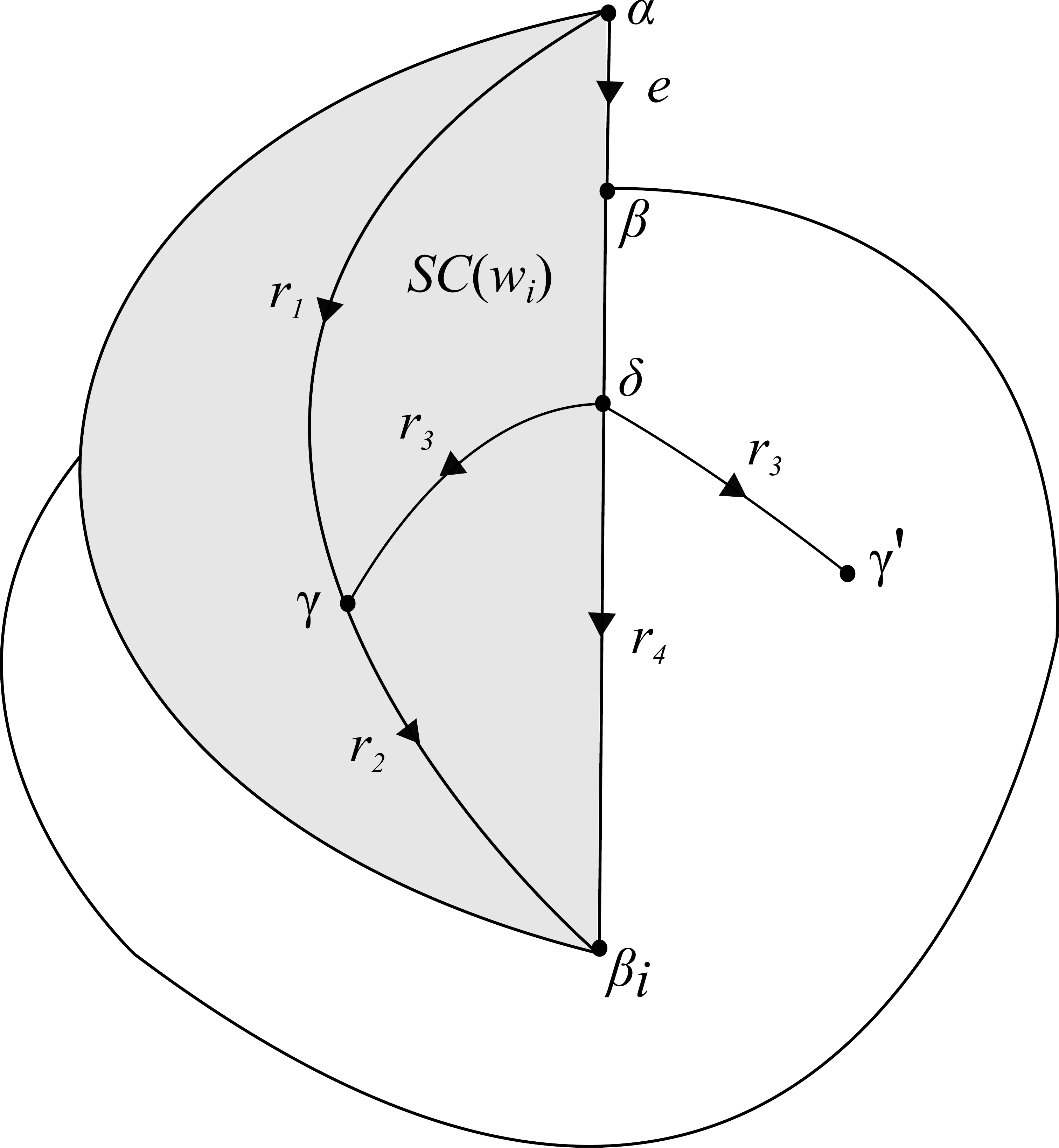}
\caption{ Case 3 for the first generation transversals of the proof of the Claim}
\label{fig53}
\end{figure}

We assume that our claim is true for all paths from $\gamma$ to $\beta_i$ that lie along a $(n-1)$-st generation transversal of $SC(w_i)$. We prove that our claim is true for all $n$-th generation transversals of $SC(w_i)$ as well.

Suppose that an arbitrary path from $\gamma$ to $\beta_i$ is labeled by the positive word $s_2$ and suppose that this path extends to an $n$-th generation transversal $t$ of $SC(w_i)$ . Let $t\equiv s_1s_2$ where $s_1$ labels the sub-path of $t$ from $\alpha$ to $\gamma$ and $s_2$ labels the sub-path of $t$ from $\gamma$ to $\beta_i$.  Again, since the vertex $\gamma$ in $SC(w_i)$ gets identified with with the vertex $\gamma'$ in $S$ through folding, we know that there is a Dyck word $ss^{-1}$ that labels a path from $\gamma$ to $\gamma'$. Since $\gamma$ lies on an $n$-th generation transversal of $SC(w_i)$, the Dyck word path $ss^{-1}$ must pass through some vertex $\delta$ of $SC(w_i)$ that lies on an $(n-1)$-st generation transversal of $SC(w_i)$. We examine the following three cases.

 \textit{Case 1.} Suppose the vertex $\delta$ is the vertex $\beta_i$ and we read the Dyck word $s_2s_2^{-1}$ in $S'_k$ from $\gamma$ to $\gamma'$.  Then we fold the oppositely oriented paths labeled by $s_2$. So in this case it is obvious that that path in $SC(w_i)$ labeled by $s_2$ gets identified with the path in $S$ from $\gamma'$ to $\beta_i$.

\textit{Case 2.} (See Figure \ref{fig54}.)   Suppose $q$ denotes an $(n-1)$-st generation transversal of $SC(w_i)$ and the vertex $\delta$ lies on the transversal $q$. Assume also that $\delta$ has already been identified with a vertex of $S$ as a consequence of folding along the Dyck word $ss^{-1}$. In this case (Case 2), we assume that the portion of the Dyck word path from $\gamma$ to $\delta$ is labeled by a positive word $s_3\in X^+$. Since $\delta$ has already been folded and identified with a vertex of $S$, we have the Dyck word $s_3s_3^{-1}$ labeling a path from $\gamma$ to $\delta$ to $\gamma'$ in the partially folded $S'_k$. The path in $SC(w_i)$ labeled by $s_3$ from $\gamma$ to $\delta$ can be extended along the $(n-1)$-st generation transversal $q$ to the vertex $\beta_i$. We assume that this path is labeled by $s_3s_4\in X^+$, where $s_4\in X^+$ labels the sub-path of $q$ from $\delta$ to $\beta_i$. By the induction hypothesis the sub-path of $q$ from $\delta$ to $\beta_i$ gets identified with a path in $S$. Thus we have a path in $SC(w_i)$ labeled by $s_3s_4$ from $\gamma$ to $\beta_i$ and we also have a path in $S$ labeled by $s_3s_4$ from $\gamma'$ to $\beta_i$.  The positive words $s_2$ and $s_3s_4$ label two co-terminal paths in $SC(w_i)$. So by Lemma \ref{bigon} an $S$-diagram corresponding to the pair of words $(s_2,s_3s_4)$ embeds in $SC(w_i)$. This $S$-diagram also embeds in $S$, because $S$ contains the path labeled by one side of the $S$-diagram (namely, $s_3s_4$) and $S$ is closed under elementary $\mathscr{P}$-expansion. So the two $S$-diagrams corresponding to the pair of words $(s_2,s_3s_4)$ get identified with each other. Hence, the path in $SC(w_i)$ labeled by $s_2$ gets identified with a path in $S$, and the claim holds in this case.

\begin{figure}[h!]
\centering
\includegraphics[trim = 0mm 0mm 0mm 0mm, clip,width=2.5in]{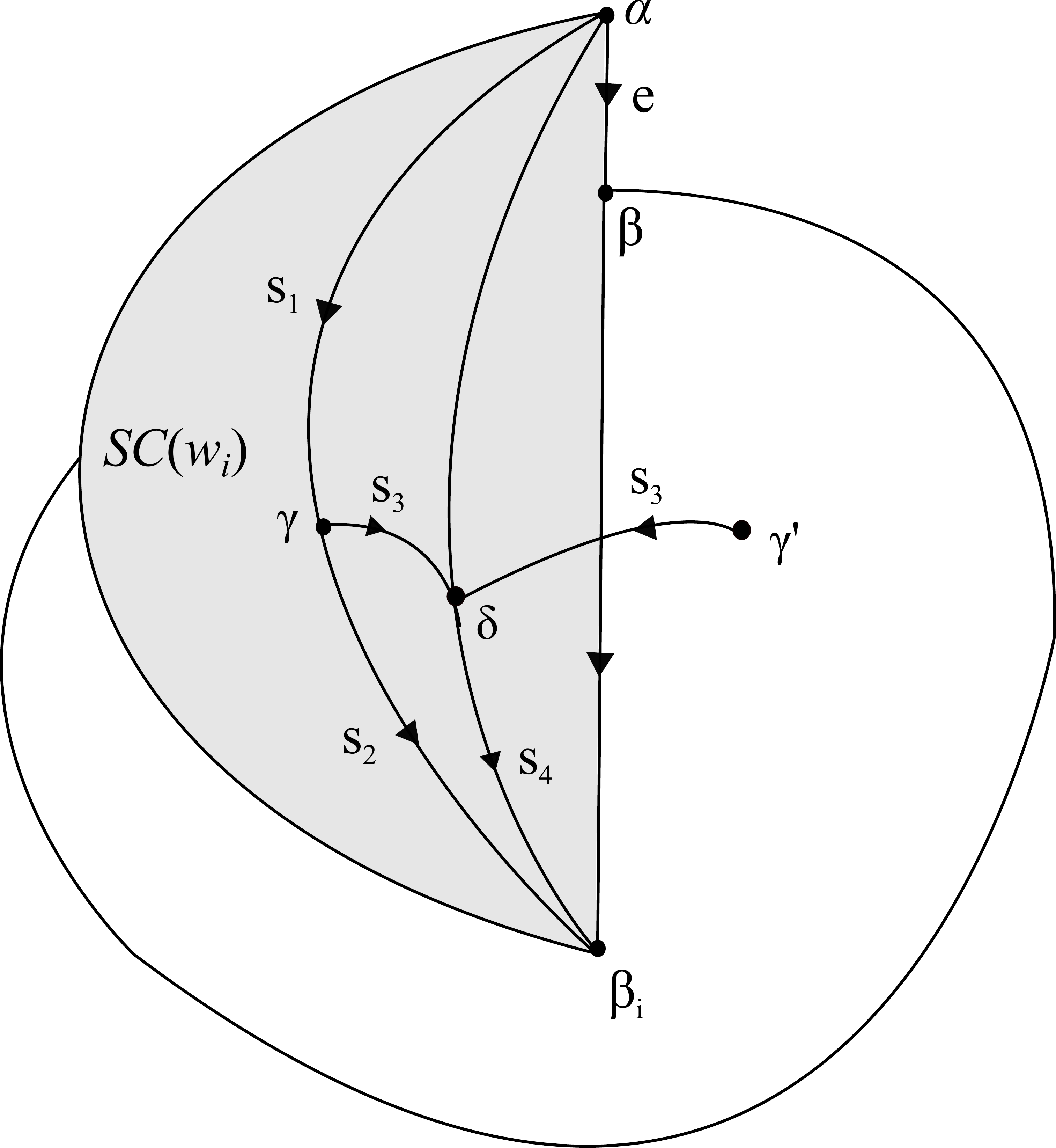}
\caption{Case 2 for the $n$-th generation transversal of the proof of the claim }
\label{fig54}
\end{figure}

\textit{Case 3.} (See Figure \ref{fig55}.) Suppose $q$ denotes an $(n-1)$-st generation transversal of $SC(w_i)$ and $\delta$ is a vertex on the transversal $q$ that has already been identified with a vertex of $S$. We also assume that there exists a path, labeled by $s_3\in X^+$, from $\delta$ to $\gamma$ such that we can read the Dyck word $s_3^{-1}s_3$ from $\gamma$ to $\delta$ to $\gamma'$ in $S'_k$.  We extend the path labeled by $s_3$ from $\delta$ to $\gamma$ along the $n$-th generation transversal $t$ to a positively labeled path (labeled by $s_3s_2\in X^+$) from $\delta$ to $\beta_i$. By the induction hypothesis the sub-path of $q$ from the vertex $\delta$ to the vertex $\beta_i$ (say, labeled by $s_4\in X^+$) is in $S$. Hence $s_3s_2$ and $s_4$ label two co-terminal paths in $SC(w_i)$ and the path labeled by $s_4$ is also in $S$. By Lemma \ref{bigon} an $S$-diagram corresponding to the pair of words $(s_3s_2,s_4)$ embeds in $SC(w_i)$. This $S$-diagram also embeds in $S$, because $S$ contains one side of this $S$-diagram and $S$ is closed under elementary $\mathscr{P}$-expansion. Thus the two $S$-diagrams corresponding to the pair of words $(s_3s_2,s_4)$ get identified with each other and the claim holds in this case as well. This completes the inductive step in the proof of the claim.

\begin{figure}[h!]
\centering
\includegraphics[trim = 0mm 0mm 0mm 0mm, clip,width=2.5in]{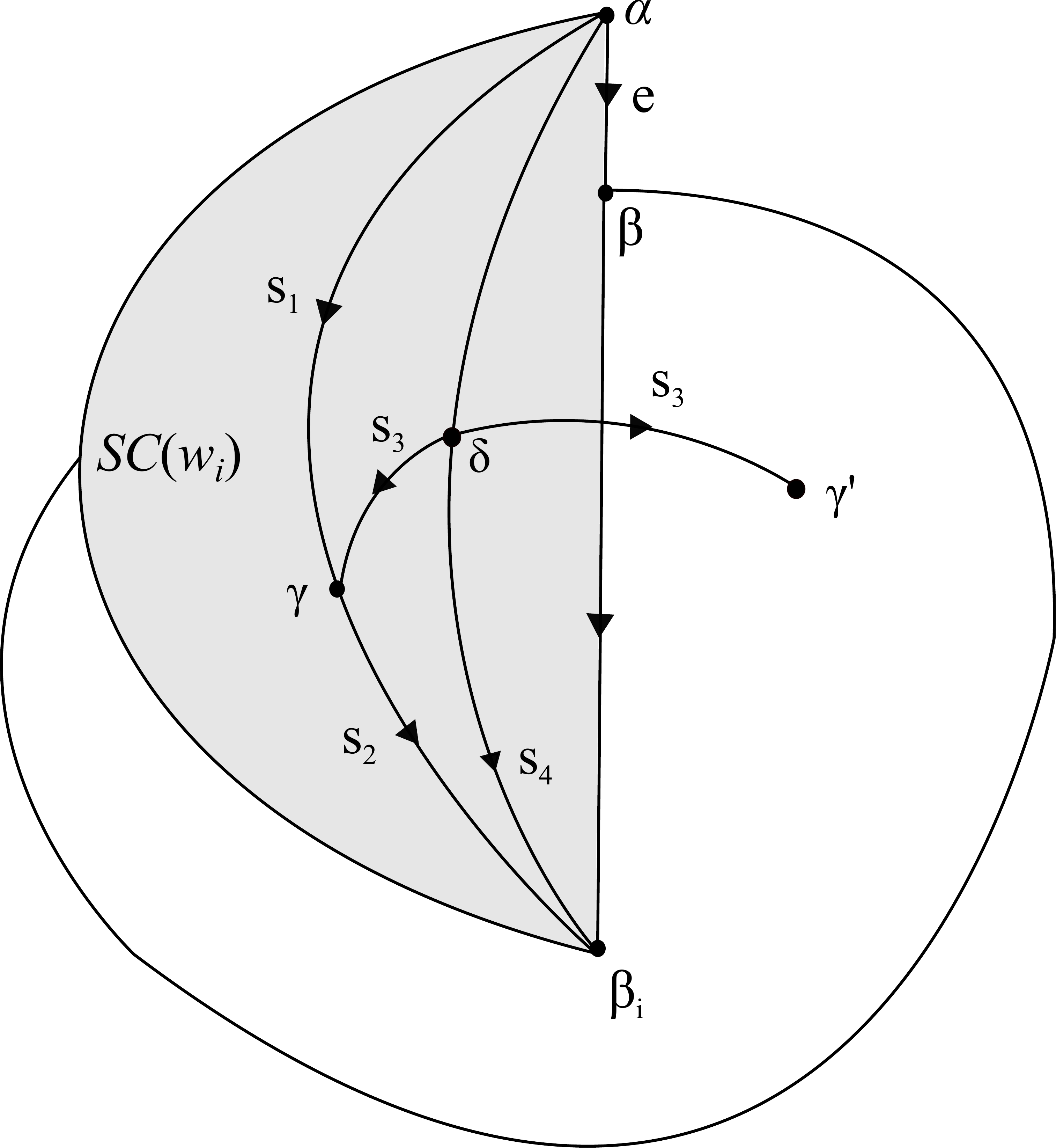}
\caption{ Case 3 for the $n$-th generation transversal of the proof of the claim}
\label{fig55}
\end{figure}

\end{proof}

We perform all possible foldings mentioned in the proof of above claim in $S'_1$ and denote the resulting complex by $S_2$. In this folding process edges of the complex $S$ are folded with the edges of $SC(w_i)$ for some $i$. This folding process may create new maximal positively labeled paths starting from $\alpha$ that did not exist in $S_1'$.

For example, $S_1'$ may contain a path labeled by $r_3^{-1}r_5$ for some $r_5\in X^+$ starting from the vertex $\beta_i$ as shown in the Figure \ref{00}. After folding the path labeled by $r_3r_3^{-1}$, we can read a new positively labeled path starting from $\alpha$ and labeled by $r_1r_5$ which could not be read in $S_1'$.
\begin{figure}[h!]
\centering
\includegraphics[trim = 0mm 0mm 0mm 0mm, clip,width=5in]{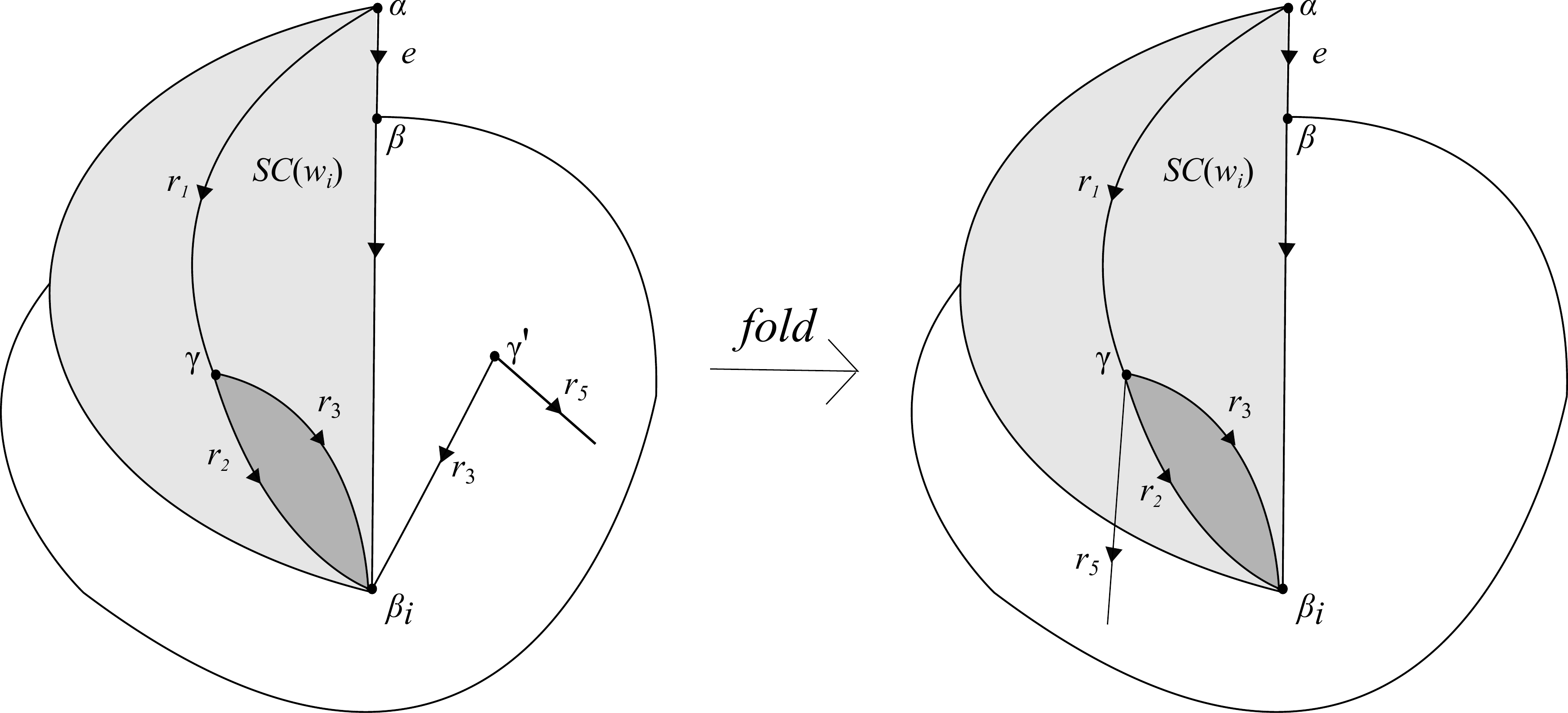}
\caption{An example showing the construction of new maximal positively labeled paths starting from $\alpha$ as a consequence of folding of edges in $S_1'$}
\label{00}
\end{figure}

It is also possible that $S_1'$ may contain a path labeled by $r_4r_5$ for some $r_5\in X^+$ starting from the vertex $\delta$ as shown in the following diagram. After folding the path labeled by $r_4^{-1}r_4$, we can read a new positively labeled path starting from $\alpha$ and labeled by $r_1r_5$ which could not be read in $S_1'$.
\begin{figure}[h!]
\centering
\includegraphics[trim = 0mm 0mm 0mm 0mm, clip,width=5in]{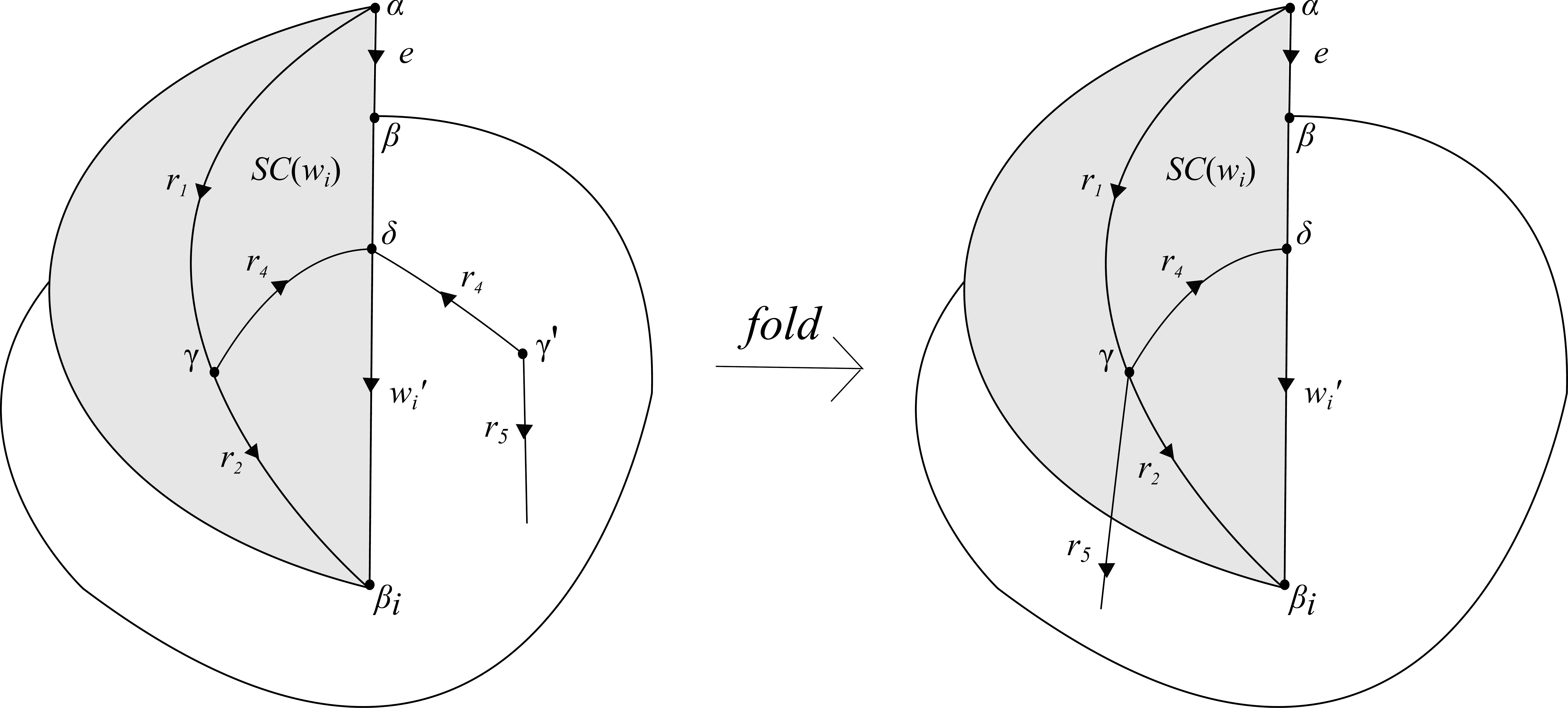}
\caption{ An example showing the construction of new maximal positively labeled paths starting from $\alpha$ as a consequence of folding of edges in $S_1'$}
\label{}
\end{figure}

  Now we consider all those maximal positively labeled paths in $S_2$ that start from $\alpha$ and did not exist in $S_1'$. Since there are only finitely many such paths, we assume that these paths are labeled by  $u_1,u_2,..., u_m$ and for each $i\in \{1,2,...,m\}$ the path labeled by $u_i$ terminates at the vertex $\gamma_i$. Note that each maximal path labeled by $u_i$ must terminate at a vertex $\gamma_i$ that is in the original complex $S$. Then $SC(u_i)$ for each $i$ is finite by hypothesis. So, we attach $SC(u_i)$ at the path labeled by $u_i$ for all $i$ and denote the resulting complex by $S_2'$.

  If the paths labeled by $u_i$ and $u_j$ for some $i\neq j$ are co-terminal, then by the same argument as above $SC(u_i)$ and $SC(u_j)$ are isomorphic to each other as edge labeled graphs. So, they get identified with each other.

If for some $i$, $SC(u_i)$ contains a transversal labeled by $t\in X^+$ such that $\theta(\neq \gamma_i)$ is the first vertex of $u$ that gets identified with a vertex in $S_2$, then the entire positively labeled sub-path of $t$ from $\theta$ to $\gamma_i$ gets identified with a positively labeled path in $S_2$. This can be verified by using the same argument as above.

We perform all possible foldings in $S_2'$ and denote the resulting complex by $S_3$. In this folding process edges of the complex $S$ are folded with the edges of $SC(u_i)$ for some $i$.This folding process may create new maximal positively labeled paths starting from $\alpha$ that did not exist in $S_2'$. So, we repeat this entire process again.

  We keep repeating this process of sewing finite Sch\"{u}tzenberger complexes of positive words and folding unless we reach to the point where we can not create new maximal positively labeled paths starting from $\alpha$ and terminating at a vertex in the  complex $S$. This process of expansion and folding eventually terminates, because $\langle X|R\rangle$ is a finite presentation, therefore an edge labeled by a letter can only be incident with finitely many 2-cells and there are only finitely many edges in $S$ as $S$ is a finite complex. By construction, the resulting complex is deterministic and closed with respect to elementary $\mathscr{P}$-expansions, so it is the Sch\"{u}tzenberger complex of $w$.  So, the Sch\"{u}tzenberger complex of $w$ is finite.

\end{proof}

 The following Corollary follows immediately from the Theorem \ref{finite complexes}, because the Sch\"{u}tzenberger complex, $SC(w)$, for any $\in(X\cup X^{-1})^*$, is finite over a presentation $(X,R)$ if and only if the underlying graph (1-skeleton) of $SC(w)$ is finite over the same presentation.

 \begin{corollary}\label{finite graphs}
Let $M=Inv\langle X|R\rangle$ be a finitely presented Adian inverse semigroup. Then the Sch\"{u}tzenberger graph of $w$ is finite for all words $w\in (X\cup X^{-1})^*$ if and only if  the Sch\"{u}tzenberger graph of $w'$ is finite for all positive words $w'\in X^+$.
\end{corollary}

 \section{Some applications of Theorem \ref{finite complexes}}

 \subsection{The word problem for a sub-family of Adian inverse semigroups that satisfy condition $(\star)$}
 \begin{definition}
 We say that a positive presentation $\langle X|R\rangle$ satisfies \textit{condition $(\star)$}, if it satisfies the following two conditions:

\begin{enumerate}
\item No proper prefix of an $R$-word is a suffix of itself or any other $R$-word.

\item No proper suffix of an $R$-word is a prefix of itself or any other $R$-word.

\end{enumerate}
\end{definition}

If $\langle X|R\rangle$ is a finite Adian presentation that satisfies condition $(\star)$, then the set of relations $R$ consists of two types of relations. First, those relations which are of the form $(u,xvy)$, where $u$ and $v$ are $R$-words and $x,y\in X^+$. Second, those relations $(u,v)\in R$ where neither $u$ nor $v$ contains an $R$-word as a proper sub-word. We construct a directed graph corresponding to an Adian presentation that satisfies condition $(\star)$ as follows. We call this graph the \textit{bi-sided graph of the presentation $\langle X|R\rangle $}. The bi-sided graph of a positive presentation is defined as follows.

\begin{definition}
The \textit{bi-sided graph} of the presentation $\langle X|R\rangle$ is a finite, directed, edge-labeled graph, denoted by $BS(X, R)$ satisfying
\begin{itemize}
\item The vertex set of $BS(X, R)$ is the set of all $R$-words.

\item To define the edge set of $BS(X, R)$, let $u, v$ be two $R$-words (where it may happen that $u$ and $v$ are the same $R$-word). There is a directed edge from the vertex $u$ to the vertex $v$ if any of the following three conditions holds:

\begin{enumerate}
\item $(u, xvy)\in R$, for some $x,y\in X^+$. In this case, the directed edge from $u$ to $v$ is labeled by the ordered pair $(x,y)$.

\item $u\equiv xvy $, for some $x,y\in X^+$ and $u$ and $v$ are distinct $R$-words. In this case, the edge from $u$ to $v$ is labeled by the ordered pair $(x,y)$.

\item If $(u,v)\in R$ is such that neither $u$ nor $v$ contains any $R$-word as a proper sub-word, then there is an edge in the bi-sided graph between $u$ and $v$, pointing in both directions. This edge is labeled by $(\varepsilon, \varepsilon)$, where $\varepsilon $ denotes the empty word.
\end{enumerate}
\end{itemize}

\end{definition}

In general, the bi-sided graph $BS(X, R)$ of an Adian presentation may contain closed paths.
\begin{Ex}
The bi-sided graph of the Adian presentation $\langle a,b|aba=b\rangle$ contains a directed closed path (cycle), namely the loop consisting of the single edge labeled by $(a, a)$ from the vertex $b$ to itself. (See Figure \ref{fig40}.) The presentation $\langle a,b|aba=b\rangle$ does not satisfy condition $(\star)$ either, because the $R$-word $aba$ has the letter $a$ as a prefix and as a suffix.
\end{Ex}

\begin{figure}[h!]
\centering
\includegraphics[trim = 0mm 0mm 0mm 0mm, clip,width=4in]{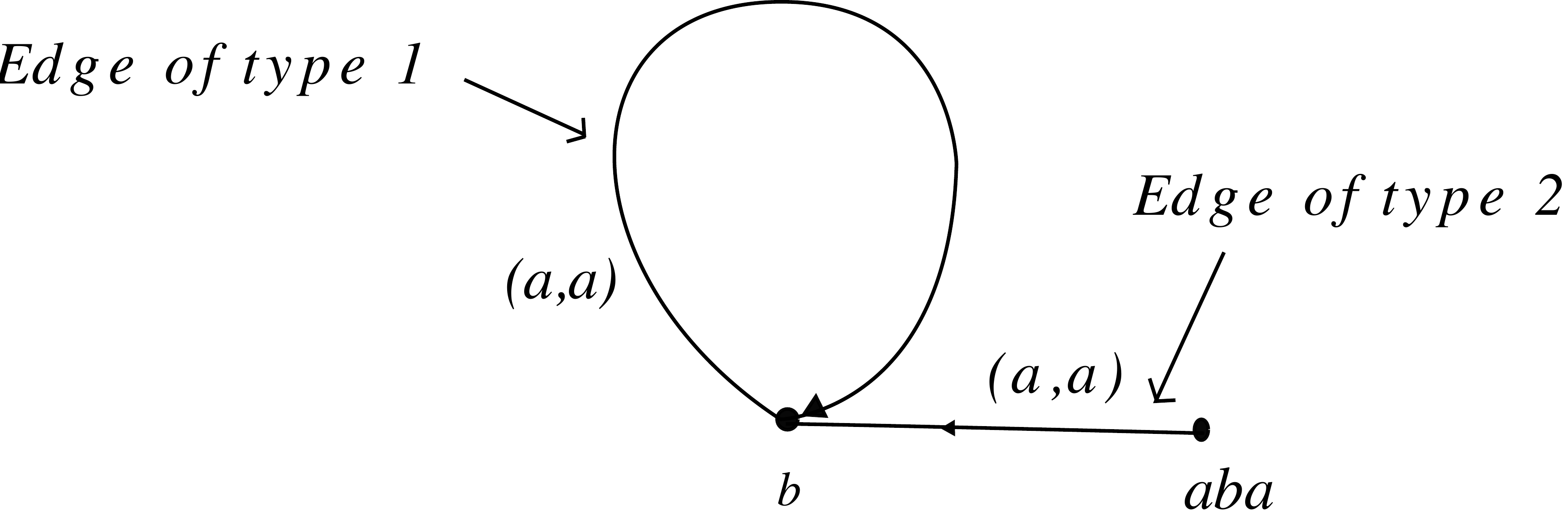}
\caption{The bi-sided graph of $\langle a,b|aba=b\rangle$ contains a cycle, (the loop at vertex $b$).}
\label{fig40}
\end{figure}

\begin{Ex}

The bi-sided graph of the Adian presentation $\langle a,b,c,d,e,f,g,h,i,j,k| a=fbg, a=jck, b=hci, c=de\rangle$ contains an undirected closed path. (See figure \ref{fig41}.) The presentation $\langle a,b,c,d,e,f,g,h,i,j,k| a=fbg, a=jck, b=hci, c=de\rangle$ satisfies condition $(\star)$.

\end{Ex}
\begin{figure}[h!]
\centering
\includegraphics[trim = 0mm 0mm 0mm 0mm, clip,width=3.5in]{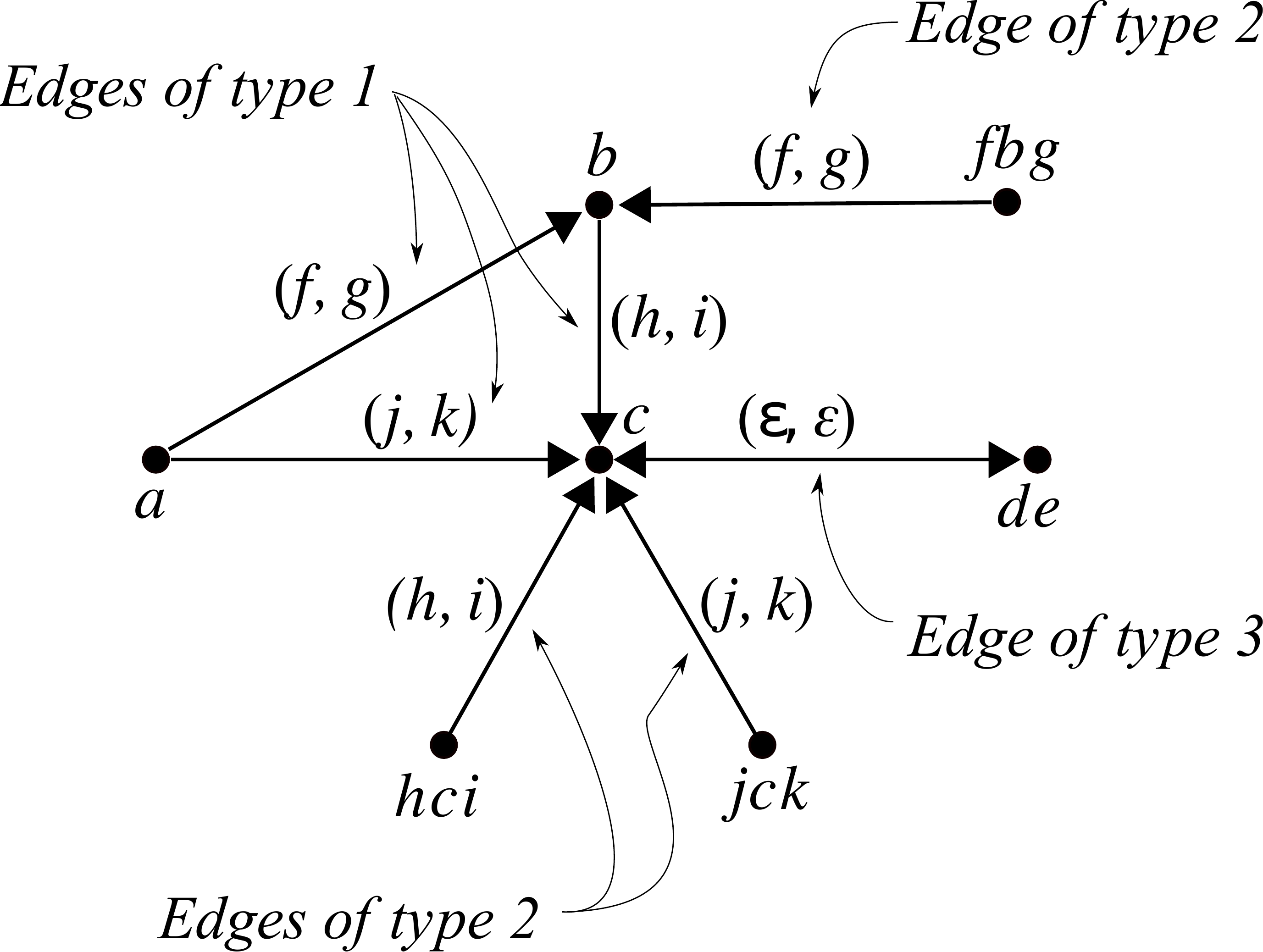}
\caption{The bi-sided graph of $\langle a,b,c,d,e,f,g,h,i,j,k| a=fbg, a=jck, b=hci, c=de\rangle$}
\label{fig41}
\end{figure}

\begin{Ex}

The bi-sided graph of the Adian presentation $\langle a,b,c,d,e,f,g,h,\\ i,j, k,l,m|a=fcg, b=hci, c=de, l=jm^2k \rangle$ is a forest. (See figure \ref{fig42}). The presentation $\langle a,b,c,d,e,f,g,h,i,j, k,l,m|a=fcg, b=hci, c=de, l=jm^2k\rangle$ satisfies condition $(\star)$.

\end{Ex}

\begin{figure}[h!]
\centering
\includegraphics[trim = 0mm 0mm 0mm 0mm, clip,width=4in]{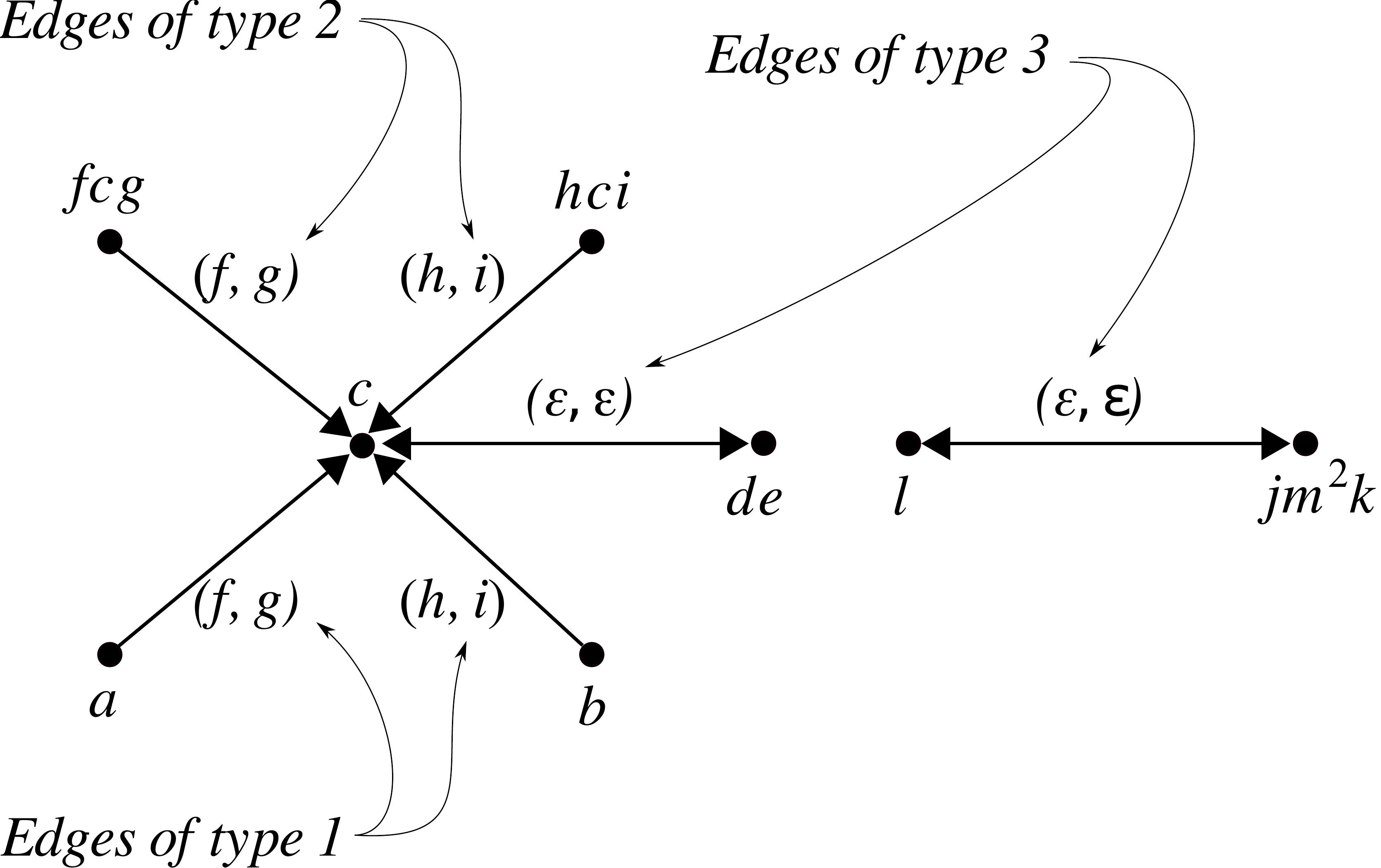}
\caption{The bi-sided graph of $\langle a,b,c,d,e,f,g,h,i,j, k,l,m|a=fcg, b=hci, c=de, l=jm^2k\rangle$}
\label{fig42}
\end{figure}

\begin{remark} For the remainder of this subsection we consider presentations $\langle X|R\rangle$ such that
\begin{enumerate}
\item $\langle X|R\rangle$ is an Adian presentation.

\item $\langle X|R\rangle$ satisfies condition $(\star)$.

\item The bi-sided graph $BS(X, R)$ is cycle-free. That is, there are no closed paths, directed or un-directed, in $BS(X, R)$. In other words, $BS(X, R)$ is a forest.
\end{enumerate}
\end{remark}

Note that if the bi-sided graph of a presentation $\langle X|R\rangle$ is a forest, then  every $R$-word labels a vertex of a connected component in $BS(X, R)$ that is a tree. For an $R$-word $u$, we let $T_u$ denote the unrooted tree that contains the vertex $u$ and we refer to $T_u$ as the \textit{bi-sided tree of $u$}. If $u$ and $v$ label two different vertices of the same bi-sided tree, then $T_u$ and $T_v$ denote the same unrooted, bi-sided tree.

\begin{definition}
Let $u$ be an $R$-word. We say that a vertex labeled by $v$ of $T_u$ is \textit{accessible} from $u$ if there exists a path labeled by $v$ in $SC(u)$.

\end{definition}

Note that if there exists a directed edge from $v_1$ to $v_2$ in the bi-sided tree of $u$ and $v_1$ is accessible from $u$ then $v_2$ is also accessible from $u$. Because if $v_1$ is accessible from $u$, then there exists a path labeled by $v_1$ in $SC(u)$. Since there exists a directed edge from $v_1$ to $v_2$ in the bi-sided tree of $u$ therefore either $(v_1,xv_2y)\in R$ or $v_1\equiv xv_2y$ for some $x,y\in X^+$. In either case there exists a path labeled by $v_2$ in $SC(u)$.

 However, if $v_2$ is accessible from $u$, and there is an edge in $BS(X, R)$ from $v_1$ to $v_2$, then $v_1$ is not necessarily accessible from $u$. Also, if a vertex $v$ of $T_u$ is \textit{not} accessible from $u$ then all of the vertices of $T_u$ that lie after the vertex $v$ going along any geodesic path from $u$ to an extremal  vertex of $T_u$ are also not accessible.

\begin{lemma}\label{complex of R-word}
Let $M=Inv\langle X|R\rangle$ be a finitely presented Adian inverse semigroup that satisfies condition $(\star)$ and (such that) $BS(X, R)$ contains no closed paths. Then $SC(u)$ is finite, for every $R$-word $u$.
\end{lemma}

\begin{proof} We start from the linear automaton $(\alpha,\Gamma_0(u),\beta)$ and obtain the approximate complex $(\alpha,\Gamma_1(u),\beta)$ by applying full $\mathscr{P}$-expansion on $(\alpha,\Gamma_0(u),\beta)$.

In the construction of $(\alpha,\Gamma_1(u),\beta)$, we attach  a path labeled by one side of a relation whose other side can be read in $(\alpha, \Gamma_0(u),\beta)$. In the linear automaton $(\alpha, \Gamma_0(u),\beta)$ we can precisely read $u$ and all those $R$-words that are proper subwords of $u$. Each relation of $R$ with one side $u$ is either of the form $(u,xvy)$ for some $x,y\in X^+$ and an $R$-word $v$ or of the form $(u,v)$ where $v$ contains no $R$-word as its proper subword. The first case corresponds to an edge of ``type 1" in $T_u$ and the second case corresponds to an edge  of ``type 3" in $T_u$. Any $R$-word $u_1$ that is a proper subword of $u$ corresponds to an edge of "type 2" in $T_u$, where $u\equiv xu_1y$ for some $x, y \in X^+$. Thus, in performing Stephen's $\mathscr{P}$-expansion to obtain $\Gamma_1(u)$ from $\Gamma_0(u)$, there is a precise correspondence between first generation transversals of $SC(u)$ and the edges of $T_u$ whose initial vertex is either labeled by $u$ or labeled by an $R$-word that is a proper subword of $u$.

If there exists an edge of type 1 or type 2 labeled by $(x,y)$ for some $x,y\in X^+$ from a vertex labeled by an $R$-word $v$ to the vertex $u$  in $T_u$, then either $(v,xuy)\in R$ or $v\equiv xuy$. We show that no generation 1 transversal contains a subpath labeled by $v$. In other words, we show that $v$ is inaccessible from $u$.

Note that we cannot read the word $xuy$ in the linear automaton $(\alpha,\Gamma_0(u),\beta)$ because $(\alpha, \Gamma_0(u),\beta)$ contains only one path from $\alpha$ to $\beta$ that is labeled by the word $u$. So, if $(v,xuy)\in R$ then we cannot attach a path labeled by $v$ to the linear automaton $(\alpha, \Gamma_0(u),\beta)$ and if $v\equiv xuy$ then we cannot read a path labeled by $v$ in  $(\alpha, \Gamma_0(u),\beta)$ because it is longer than the path labeled by $u$.

We obtain the approximate graphs $(\alpha,\Gamma_2(u),\beta)$ by applying the full $\mathscr{P}$-expansion on $(\alpha,\Gamma_1(u),\beta)$. We observe that the second generation transversals of $SC(u)$ are obtained as a consequence of attaching paths labeled by those $R$-words which label the terminal vertices of those edges of $T_u$ whose initial vertex is either a first generation transversal or a proper subword of an $R$-word that labels a first generation transversal of $SC(u)$. If there exists an edge in $BS(X, R)$ labeled by $(x_1,y_1)$ for some $x_1,y_1\in X^*$ with initial vertex labeled by $v_1$ and terminal vertex labeled by $v_2$ such that $x_1v_2y_1$ is a subword of an $R$-word that labels a first generation transversal, then we obtain a second generation transversal by sewing on a path labeled by $v_1$ from the initial vertex  to the terminal vertex of the  of the path $x_1v_2y_1$.

If there exists an edge of type 1 or type 2 labeled by $(x_1,y_1)$ with initial vertex $v_1$ and terminal vertex $v_2$ (i.e. either $(v_1,x_1v_2y_1)\in R$ or $v_1\equiv x_1v_2y_1$ ) such that $v_2$ is a subword of an $R$-word that labels a first generation transversal but $x_1v_2y_1$ is not a subword of that transversal, then we cannot attach a path labeled by $v_1$ to $(\alpha,\Gamma_1(u),\beta)$. So, none of the vertices of $T_u$ that occur after the vertex $v_1$ going along a path from the vertex $u$ to an extremal vertex of $T_u$ will be accessible from $u$ in $SC(u)$.

We observe that when we apply full $\mathscr{P}$-expansion on $(\alpha,\Gamma_n(u),\beta)$ for some $n\in \mathbb{N}$, we cover some more vertices of $T_u$ that were not covered before in the sense that we add some new transversals that contains an $R$-word that labels a vertex of $T_u$ and that $R$-word does not label a path in  $(\alpha,\Gamma_n(u),\beta)$. Since $T_u$ is a finite tree and none of the $R$-word label two distinct vertices of $T_u$ therefore the process of applying full $\mathscr{P}$-expansion must terminate after a finite number of steps. Hence $SC(u)$ is finite.
\end{proof}

\begin{lemma}\label{prefix l}
 Let $M=Inv\langle X|R\rangle$ be an Adian inverse semigroup. Let $u$ be an $R$-word and $z$ labels a proper suffix of a transversal $p$ of $SC(u)$. Then either
\begin{enumerate}

\item[(i)] $z$ contains an $R$ word that also labels a subpath of $p$, or

\item[(ii)] A prefix of $z$ is a suffix of some $R$-word that also labels a subpath of $p$.
\end{enumerate}
\end{lemma}

\begin{proof} If $z$ does not contain an $R$-word that also labels a subpath of the transversal $p$ then the initial vertex of the path labeled by $z$ is not the initial vertex of any $R$-word that labels a subpath of $p$. Hence the initial vertex of $z$ lies between a pair of vertices of $p$ that are the  initial and the terminal vertex of a subpath of $p$ that is labeled by an $R$-word. Hence a prefix of $z$ is a suffix of an $R$-word.

\end{proof}

We also remark that a dual statement also holds for a prefix of a transversal of the Sch\"{u}tzenberger complex of an $R$-word over an Adian presentation.

\begin{remark}\label{prefix r}
In Lemma \ref{prefix l}, if the presentation $\langle X|R\rangle$ satisfies condition $(\star)$ and $z$ happens to be a prefix of an $R$-word, then only $(i)$ holds, because $(ii)$ violates the condition $(\star)$.

Similarly, in the dual statement to Lemma \ref{prefix l}, if the presentation $\langle X|R\rangle$ satisfies condition $(\star)$ and $z$ happens to be a suffix of an $R$-word, then only $(i)$ holds, because $(ii)$ violates the condition $(\star)$.
\end{remark}

\begin{theorem} Let $M=Inv\langle X|R\rangle$ be a finitely presented Adian inverse semigroup that satisfies condition $(\star)$ and $BS(X, R)$ contains no closed path. Then $SC(w)$ is finite, for all $w\in (X\cup X^{-1})^*$.
\end{theorem}
\begin{proof}
We just need to show that $SC(w)$ is finite  for all $w\in X^+$ and then the above theorem follows from Theorem \ref{finite complexes}.  So, we assume that $w\in X^+$ and we construct the linear automaton of $w$, $(\alpha,\Gamma_0(w),\beta)$.

It follows from condition $(\star)$, that no two distinct $R$-words will overlap with each other. However, an $R$-word can be a proper subword of another $R$-word. So, we can uniquely factorize $w$ as $x_0u_1x_2u_2...u_nx_n$, where $x_i\in X^*$ and $u_i$'s are maximal $R$-words in the sense that none of the $u_i$'s are properly contained in another $R$-word that is also a subword of $w$.

It follows from Lemma \ref{complex of R-word}, that $SC(u_i)$ is finite for all $1\leq i\leq n$. So, we attach $SC(u_i)$ for all $1\leq i\leq n$ to the corresponding paths labeled by $u_i$'s in $(\alpha,\Gamma_0(w),\beta)$ to construct $SC(w)$ and denote the resulting complex by $S_1$. It follows from Lemma \ref{no folding}, that no two edges get identified with each other as a consequence of attaching $SC(u_i)$ for all $1\leq i\leq n$ to the linear automaton $(\alpha,\Gamma_0(w),\beta)$. If $S_1$ is closed under elementary $\mathscr{P}$-expansion, then we are done. Otherwise, we will be able to read a finite number of $R$-words labeling the paths of $S_1$ where we can attach new 2-cells  by sewing on paths labeled by the other sides of the corresponding relations.

We assume that $v_1,v_2,...,v_m$ are the $R$-words that label the paths of $S_1$ where we can attach new 2-cells. Note that each of $v_i$ labels a vertex of $T_{u_j}$ for some $i$ and $j$, that was inaccessible from $u_j$ earlier. Because if $v_i$ labels a path in $S_1$ then by Lemma \ref{prefix l} and Remark \ref{prefix r} the path labeled by $v_i$ contains an $R$-word $r_j$ as a proper subword such that the $R$-word $r_j$ labels a path in $SC(u_j)$. In other words, $r_j$ labels an accessible vertex of $T_{u_j}$ from the vertex $u_j$. Since $r_j$ is a proper subword of $v_i$, therefore there exists an edge of type 2 in $T_{u_j}$ with initial vertex labeled by $v_i$ and the terminal vertex labeled by $r_j$. So, $T_{v_i}$ and $T_{u_j}$ represent the same tree for some $i$ and $j$. By Lemma \ref{complex of R-word}, $SC(v_i)$ is finite for all $1\leq i\leq m$ and cover some more vertices of $T_{u_j}$ in the sense that $SC(v_i)$ contain paths labeled by those $R$-word which also label some of the vertices of $T_{u_j}$, for some $1\leq j\leq n$, that were not covered by $SC(u_j)$.

We attach $SC(v_i)$  to the paths labeled by $v_i$ for all $1\leq i\leq m$ in $S_1$ and denote the resulting complex by $S_2$. No two edges get identified with each other in $S_2$ as a consequence of attaching $SC(v_i)$'s  to $S_1$ by Lemma \ref{no folding}. If $S_2$ is closed under elementary $\mathscr{P}$-expansion then we are done. Otherwise we repeat this process of attaching Sch\"{u}tzenberger complexes of $R$-words and capturing more vertices of the trees $T_{u_i}$ for some $1\leq i\leq n$. This process eventually terminates, because, each $T_{u_i}$ is a finite tree with all the vertices labeled by distinct $R$-words and every $R$-word labels a vertex of exactly one tree. Hence, $SC(w)$ is a finite complex.

\end{proof}

\begin{remark}
Let $M=Inv\langle X|R\rangle$ be an Adian inverse semigroup that satisfies condition $(\star)$ and $BS(X,R)$ contains no closed path then the word problem for $M$ is decidable. It follows from the Theorem \ref{finite graph} that the Sch\"{u}tzenberger complex of every word $w\in (X\cup X^{-1})^*$ is finite over the presentation $\langle X|R\rangle$. So, for any two given words $w_1,w_2\in (X\cup X^{-1})^*$, we can easily check whether $w_1\in L(w_2)$ and $w_2\in L(w_1)$ or not.
\end{remark}


\subsection{The word problem for Inverse semigroups given by the presentation $\langle a,b|ab^m=b^na\rangle$}

In this section we show that the word problem is decidable for the inverse semigroup given by the presentation $M=Inv\langle a,b|ab^m=b^na\rangle$, where $m,n\in\mathbb{N}$.  The word problem for the case case $m=n$ follows from Corollary 6.6 of \cite{SG}. So, throughout this section we assume that $m>n$. The case $m<n$ follows from a dual argument. We can get an alternate proof for the case $n=m$ by following along same lines as in the case of $m<n$.

\begin{lemma}\label{lemma 12}
The Sch\"{u}tzenberger complex of a word $a^kb^t$ for $k,t\in \mathbb{N}$, over the presentation $\langle a,b|ab^m=b^na\rangle$, is finite.
\end{lemma}
\begin{proof}
We adopt a slightly different approach to construct $SC(a^kb^t)$. We draw edges labeled by $a$ horizontally and edges labeled by $b$ vertically. Then the linear automaton of $a^kb^t$, $(\alpha_0,\Gamma_0(a^kb^t),\beta_0)$, is shown in Figure \ref{01}. If $t<m$, then we cannot attach any 2-cell to  $(\alpha_0,\Gamma_0(a^kb^t),\beta_0)$. So, the above lemma is true for this case.

If $t\geq m$, then $t=q_1m+r_1$, where $q_1$ is the quotient and $r_1$ is a remainder and $0\leq r_1<m$. We can attach $q_1$ 2-cells in the first column along the vertical segment labeled by $b^t$ of $(\alpha_0,\Gamma_0(a^kb^t),\beta_0)$. After attaching all the 2-cells in the first column along the vertical segment labeled by $b^t$ we have created a new vertical segment labeled by $b^{nq_1}$, because there are total $q_1$ 2-cells and each 2-cells contains exactly $n$ edges on the newly attached side of the 2-cell.

\begin{figure}[h!]
\centering
\includegraphics[trim = 0mm 0mm 0mm 0mm, clip,width=3in]{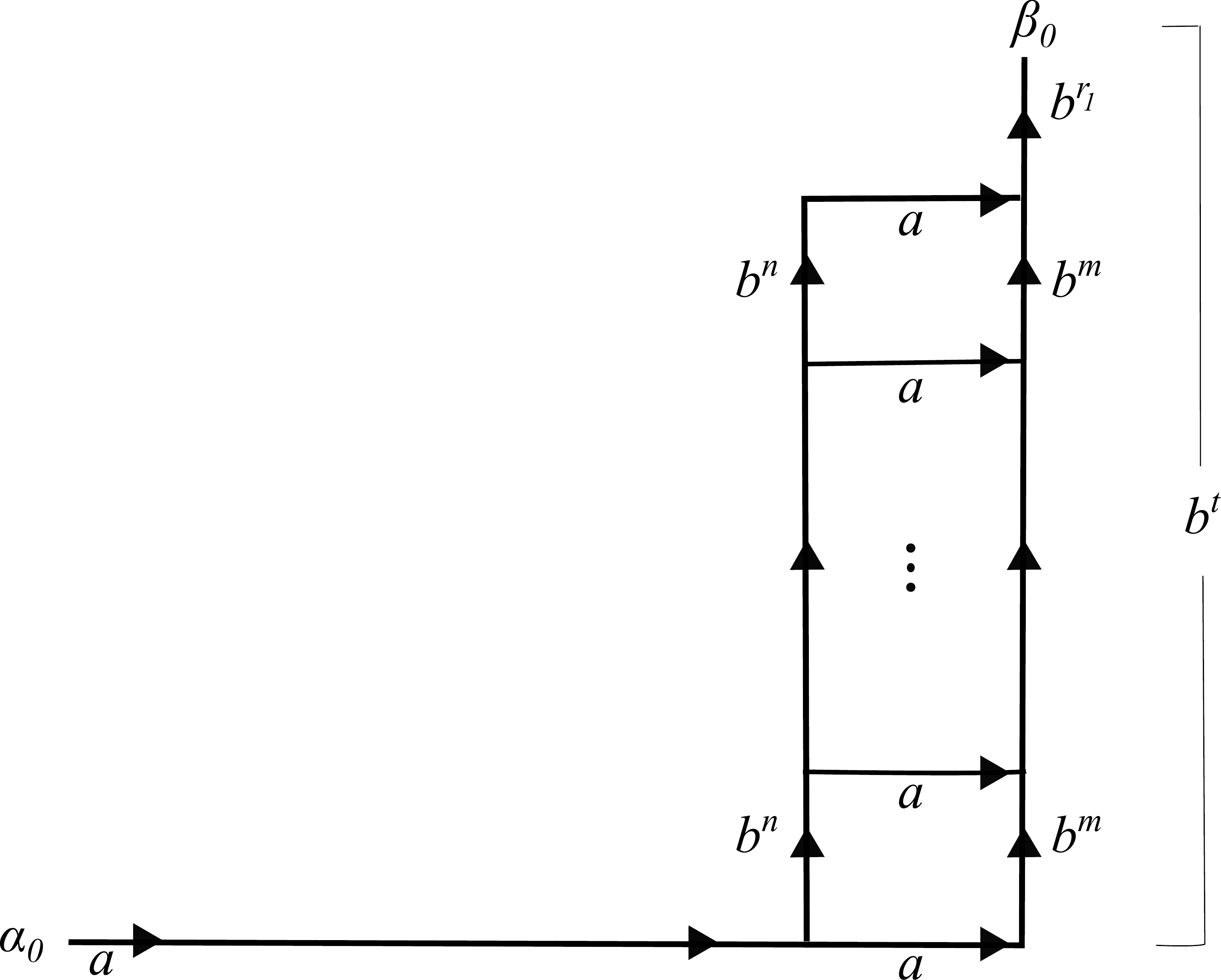}
\caption{Construction of $SC(a^kb^t)$. }
\label{01}
\end{figure}

If $nq_1 <m$ or $k=1$ then the process of attaching new 2-cells terminates at this stage. If neither $nq_1<m$ nor $k=1$, then $nq_1=q_2m+r_2$, where $q_2$ is a quotient, $r_2$ is a remainder and $0\leq r_2 < m$. So,  We can attach a column of $q_2$ 2-cells along the vertical segment labeled by $b^{nq_1}$. This process of attaching columns of new 2-cells terminates after at most $k$ steps. So, $SC(a^kb^t)$ is a finite complex.
\end{proof}
\begin{remark}
In the above construction of $SC(a^kb^t)$ in Lemma \ref{lemma 12}, every new maximal vertical segment contains fewer edges than the other vertical side of the same column of 2-cells.

\end{remark}
\begin{lemma}\label{lemma 13}
The Sch\"{u}tzenberger complex of a word $b^ta^k$ for $k,t\in \mathbb{N}$, over the presentation $\langle a,b|ab^m=b^na\rangle$, is finite.
\end{lemma}
\begin{proof}
We draw edges labeled by $a$ horizontally and edges labeled by $b$ vertically. Then the linear automaton of $b^ta^k$, $(\alpha_0,\Gamma_0(b^ta^k),\beta_0)$, is shown in Figure \ref{02}. If $t<n$, then we cannot attach any 2-cell to  $(\alpha_0,\Gamma_0(b^ta^k),\beta_0)$. So, the above lemma is true for this case.

If $t\geq n$, then $t=q_1n+r_1$, where $q_1$ is the quotient and $r_1$ is a remainder and $0\leq r_1<n$. We can attach $q_1$ 2-cells in the first column along the vertical segment labeled by $b^t$ of $(\alpha_0,\Gamma_0(a^kb^t),\beta_0)$. After attaching all the 2-cells in the first column along the vertical segment labeled by $b^t$ we have created a new vertical segment labeled by $b^{mq_1}$, because there are total $q_1$ 2-cells and each 2-cells contains exactly $m$ edges on the newly attached side of the 2-cell.

\begin{figure}[h!]
\centering
\includegraphics[trim = 0mm 0mm 0mm 0mm, clip,width=3in]{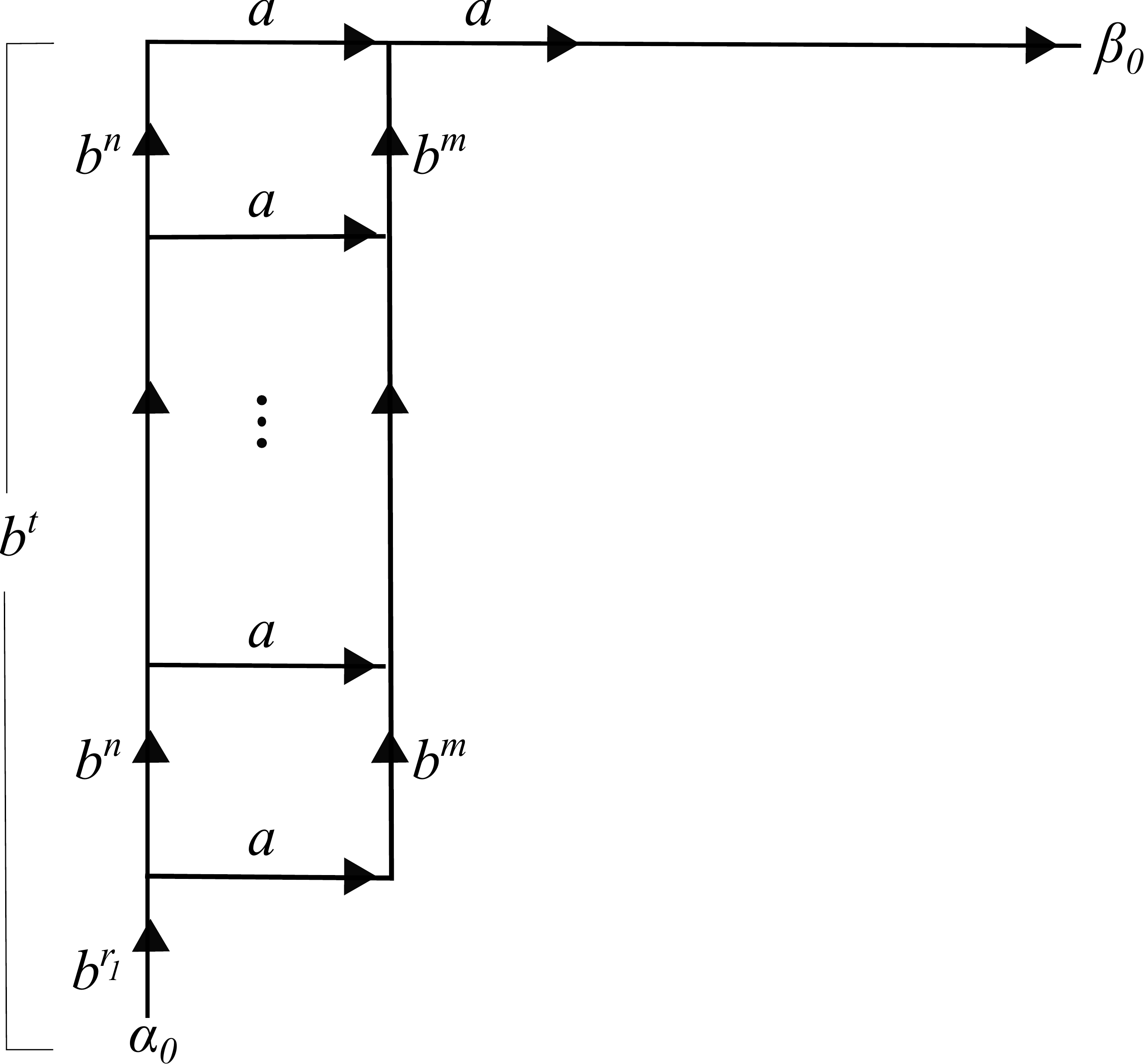}
\caption{Construction of $SC(b^ta^k)$.}
\label{02}
\end{figure}

If $k=1$ then the process of attaching new 2-cells terminates at this stage. Otherwise $mq_1=q_2n+r_2$, where $q_2$ is a quotient, $r_2$ is a remainder and $0\leq r_2 < n$. So,  We can attach a column of $q_2$ 2-cells along the vertical segment labeled by $b^{mq_1}$. Clearly, this process of attaching columns of new 2-cells terminates after $k$ steps. So, $SC(b^ta^k)$ is a finite complex.

\end{proof}
\begin{remark}
In the above construction of $SC(b^ta^k)$ in Lemma \ref{lemma 13} every new maximal vertical segment contains more edges than the other vertical side of the same column of 2-cells.

\end{remark}
\begin{theorem} \label{421}  For all $w\in \{a,b,a^{-1},b^{-1}\}^*$ the Sch\"{u}tzenberger complex of a word $w$, $SC(w)$, over the presentation $Inv\langle a,b|ab^m=b^na \rangle $ is finite. Hence the word problem is decidable for $M$.
\end{theorem}

\begin{proof}
  Since $\langle X|R\rangle$ is an Adian presentation, we just need to show that $SC(w)$ for all $w\in \{a,b\}^+$ is finite, then Theorem \ref{421} follows from Theorem \ref{finite complexes}.

If $w$ is of the form $a^k$ or $b^k$ for some $k\in \mathbb{N}$, then $SC(w)$ is finite.

We assume that $w \equiv a^{k_0}b^{t_0}a^{k_1}b^{t_1}...a^{k_l}b^{t_l}$, where $k_0,t_l\in \mathbb{N}\cup\{0\}$ and $k_i,t_j\in \mathbb{N}$ for $1\leq i\leq l$ and $0\leq j\leq l-1$.

We construct $SC(w)$ by drawing the edges labeled by $a$ horizontally and the edges labeled by $b$ vertically.  So, $(\alpha_0,\Gamma_0(w),\beta_0)$ looks like the diagram shown in Figure \ref{03}. We attach $SC(a^{k_i}b^{t_i})$ on the path labeled by $a^{k_i}b^{t_i}$ of $(\alpha_0,\Gamma_0(w),\beta_0)$ wherever it is possible to attach and we denote the resulting complex by $S_1$.  No two edges get identified with each other as a consequence of attaching these finite complexes to $(\alpha_0,\Gamma_0(w),\beta_0)$ by Lemma \ref{no folding}. As a consequence of attaching these finite complexes to the $(\alpha_0,\Gamma_0(w),\beta_0)$, we have created at most $l-1$ new maximal directed paths labeled by $a^kb^t$ for some $k,t\in \mathbb{N}$.
\begin{figure}[h!]
\centering
\includegraphics[trim = 0mm 0mm 0mm 0mm, clip,width=3in]{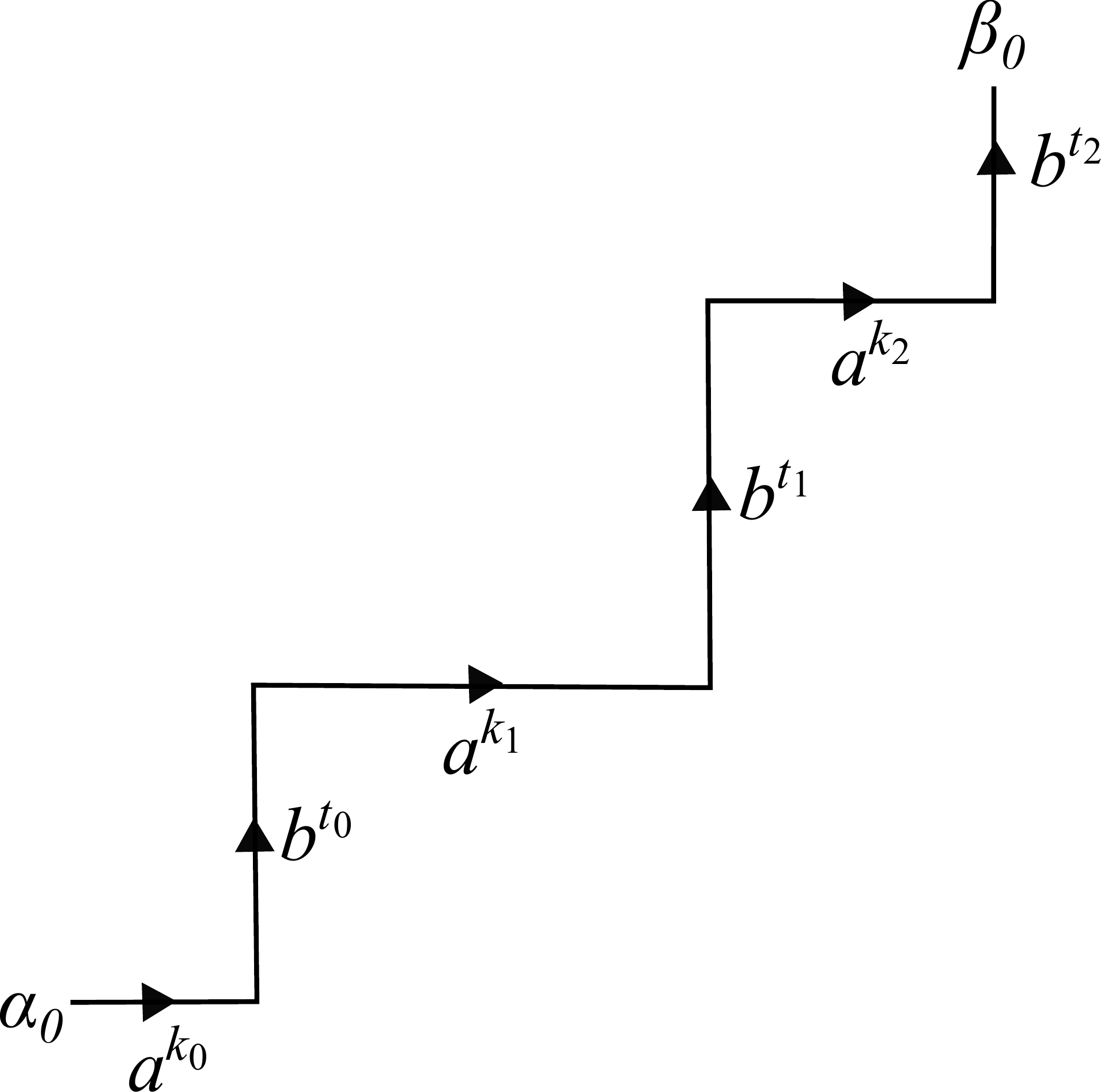}
\caption{$(\alpha_0,\Gamma_0(w),\beta_0)$}
\label{03}
\end{figure}

We attach finite complexes of the form $SC(a^kb^t)$ to every new maximal path labeled by a word of the form $a^kb^t$ for some $k,t\in \mathbb{N}$ in $S_1$. We denote the resulting complex by  $S_2$. Again by Lemma \ref{no folding}, no two edges get identified with each other in $S_2$. We can read at most $l-2$ new maximal directed paths in $S_2$ which are labeled by the words of the form $a^kb^t$ for some $k,t\in \mathbb{N}$. So, we repeat the process attaching finite Sch\"{u}tzenberger complexes of the words of the form $a^kb^t$ where ever it is possible to attach and denote the resulting complex by $S_3$. Note that this process of attaching finite Sch\"{u}tzenberger complexes of the words of the form $a^kb^t$ eventually terminates after at most $l$ steps. We denote the resulting complex by $S'$.

Now in $S'$, on the other side of the path labeled by $w$, we attach finite Sch\"{u}tzenberger complexes of the words $b^{t_i}a^{k_{i+1}}$  at the paths labeled by $b^{t_i}a^{k_{i+1}}$  for $0\leq i\leq l-1$, where ever it is possible to attach and denote the resulting complex by $S'_1$.  By Lemma \ref{no folding} no two edges in $S'_1$ get identified with each other as a consequence of attaching these finite complexes. As a consequence of attaching these finite complexes we have created at most $l-1$ new maximal paths which are labeled by the words of the form $b^ta^k$ for some $k,t\in \mathbb{N}$. So, we repeat the process of attaching finite Sch\"{u}tzenberger complexes of the words of the form $b^ta^k$  at the corresponding new paths in $S'_1$. We denote the resulting complex by $S'_2$. This process of attaching finite complexes of the words of the form $b^ta^k$ terminates after at most $l$ steps and we obtain a finite complex which is closed under elementary $\mathscr{P}$-expansion and folding. Hence, $SC(w)$ is a finite complex.

 \end{proof}

\section*{Acknowledgement}
The author of this paper is thankful to John Meakin and Robert Ruyle for their several useful suggestions.

\bibliographystyle{plain}

\end{document}